\documentclass[a4paper,12pt,english]{article}
\usepackage{hyperref}
\usepackage{babel}
\usepackage[latin1]{inputenc}
\usepackage{amsmath,amsthm,amssymb,amsfonts}
\usepackage{enumerate,enumitem}
\usepackage{graphics,graphicx,subfigure,epsfig,psfrag}
\usepackage{mathrsfs,marginnote} 
\usepackage{xcolor}

\evensidemargin = \oddsidemargin
\advance\textwidth by -2in
\theoremstyle{plain}
\newtheorem{theoreme}{Theorem}[section]
\newtheorem{theoreml}{Theorem}

\newtheorem{theoremp}{Theorem}

\newtheorem{theoremc}{Theorem}

\newtheorem{theoremr}{Theorem}

\newtheorem{theoremrb}{Theorem}

\newtheorem{corollaire}[theoreme]{Corollary}

\newtheorem{proposition}[theoreme]{Proposition}

\newtheorem{lemme}[theoreme]{Lemma}

\theoremstyle{definition}

\theoremstyle{remark}
\newtheorem{remarque}[theoreme]{Remark}

\def\N{{\mathbb N}}
\def\Q{{\mathbb Q}}
\def\R{{\mathbb R}}
\def\Z{{\mathbb Z}}
\def\T{{\mathbb T}}
\def\cT{{\mathcal T}}
\def\cA{{\mathcal A}}
\def\cF{{\mathcal F}}
\def\cB{{\mathcal B}}

\def\valpha{{\langle\alpha\rangle}}
\def\eps{\varepsilon}
\def\f{\varphi}

\def\from{\colon} 

\def\Diff{{\mathrm {Diff}}}
\def\St{{\mathrm {Stab}}}
\def\Homeo{{\text{\rm Homeo}}}
\def\Homeot{{\text{\rm H}(\T^1)}}
\def\D{{\mathrm{D}}}
\def\CC{{C}}
\def\tq{{\tilde q}}

\def\cS{{\mathcal S}}
\def\FF{{\mathcal F}}

\def\Cinf{{C^{\infty}}}

\def\ie{{\emph{i.e.\ }}}

\newcommand{\id}{\mathop{\mathrm{id}}\nolimits}
\def\Fix{\mathop{\mathrm{Fix}}}

\def\ITI{\mathop{\mathrm{ITI}}}

\def\norm#1{\left\lVert #1 \right\rVert}
\def\rep#1{\{ #1 \}}

\def\lrbars#1{\left\lvert #1 \right\rvert} 
\def\Bars#1{\lVert #1 \rVert} 
\def\bigBars#1{\bigl\lVert #1 \bigr\rVert}

\def\up{\textup}

\title{Smooth times of a flow in dimension $1$}

\begin{document}

\author{H\'el\`ene~Eynard-Bontemps}

\date{}

\maketitle

\begin{abstract}
Let $\alpha$ be an irrational number and $I$ an interval of $\R$. If $\alpha$ is \emph{Diophantine}, we show that any one-parameter group of homeomorphisms of $I$ whose time-$1$ and time-$\alpha$ maps are $C^\infty$ is in fact the flow of a $C^\infty$ vector field. If $\alpha$ is \emph{Liouville} on the other hand, we construct a one-parameter group of homeomorphisms of $I$ whose time-$1$ and time-$\alpha$ maps are $C^\infty$ but which is not the flow of a $C^2$ vector field (though, if $I$ has boundary, we explain that the hypotheses force it to be the flow of a $C^1$ vector field). We extend both results to \emph{families} of irrational numbers, the critical arithmetic condition in this case being \emph{simultaneous ``diophantinity''}. 

For one-parameter groups defining a \emph{free} action of $(\R,+)$ on $I$, these results follow from famous linearization theorems for circle diffeomorphisms. The novelty of this work concerns non-free actions. \bigskip

\begin{center}
{\large Temps lisses d'un flot en dimension $1$}\bigskip

\textbf{R\'esum\'e}
\end{center}

Soit $\alpha$ un nombre irrationnel et $I$ un intervalle de $\R$. Si $\alpha$ est \emph{diophantien}, on montre que tout groupe \`a un param\`etre d'hom\'eomorphismes de $I$ dont les temps $1$ et $\alpha$ sont de classe $C^\infty$ est en fait le flot d'un champ de vecteurs $C^\infty$. Si au contraire $\alpha$ est \emph{de Liouville}, on construit un groupe \`a un param\`etre d'hom\'eomorphismes de $I$ dont les temps $1$ et $\alpha$ sont de classe $C^\infty$ mais qui n'est pas le flot d'un champ de vecteurs $C^2$ (toutefois, si $I$ a un bord non vide, on explique qu'il s'agit automatiquement du flot d'un champ $C^1$). On \'etend ces deux r\'esultats \`a des familles de nombres irrationnels, la condition arithm\'etique critique \'etant dans ce cas le caract\`ere ``simultan\'ement diophantien''. 

Pour des groupes \`a un param\`etre d\'efinissant une action \emph{libre} de $(\R,+)$ sur $I$, ces r\'esultats d\'ecoulent de c\'el\`ebres th\'eor\`emes de lin\'earisation pour les diff\'eomorphismes du cercle. La nouveaut\'e de ce travail concerne les actions non libres.

\end{abstract}
\bigskip

\begin{footnotesize}
\noindent\emph{Key words}: Homeomorphisms, diffeomorphisms, vector fields and flows in dimension $1$, centralizer, Diophantine and Liouville numbers. 
\medskip

\noindent\emph{Mots clefs} : Hom\'eomorphismes, diff\'eomorphismes, champs de vecteurs et flots en dimension $1$, centralisateur, nombres diophantiens et de Liouville. 
\medskip

\noindent\emph{2010 Mathematics Subject Classification:}  37C05, 37C10, 37E05 (37C15, 37E10, 37E45).
\end{footnotesize}\bigskip

{\small \noindent\emph{Acknowledgement}. I wish to dedicate this article to the memory of Jean-Christophe Yoccoz, who suggested to me, as I was still a PhD student, a ``program'' which led to the first ``half'' (Section \ref{s:Diophantine}) of this article many years later. He was right, it was ``tractable''. I also wish to thank Christian Bonatti, Sylvain Crovisier and Bassam Fayad for everything they taught me about one-dimensional dynamics (and Bassam for his precious advice on this particular subject), as well as the organizers of the ``Aussois winter school of geometry and dynamics'', whose invitation motivated me to dive back into this long-term project.
\medskip

\section{Introduction}
\label{s:intro}

\textbf{Standing assumptions and vocabulary.} In this article, all vector fields are assumed time-independent, of regularity at least $C^1$, and complete. By the \emph{flow} of such a vector field on an interval $I$, we mean the one-parameter group of diffeomorphisms of $I$ made of its time-$t$ maps, for $t$ in $\R$. We will sometimes refer to a one-parameter group of homeomorphisms of $I$ as a ``$C^0$ flow'' on $I$, and sometimes confuse such a group with the continuous $\R$-action that it defines on $I$, thus talking about the ``time-$t$ maps of the action''.

\subsection{Motivation and results}
\label{ss:results}

This work is initially motivated by the study of smooth $\Z^2$-actions on the segment and their possible deformations, in relation with codimension-one foliations of $3$-manifolds (cf.\ \cite{Ey16}). In order to manipulate such actions, one can try to describe the \emph{centralizer} of a given element of the group of smooth diffeomorphisms of the segment. To that end, one first needs to understand the local picture near an isolated fixed point. This motivates the study of (the centralizer of) $C^\infty$-diffeomorphisms $f$ of $\R_+$ without fixed points in $\R_+^*=(0,+\infty)$. These divide into \emph{contractions} and  \emph{expansions} satisfying $f(x)<x$ (resp. $f(x)>x$) for every $x\in\R_+^*$. Since it is lighter to mention ``contractions'' rather than ``diffeomorphisms of $\R_+$ without fixed points in $\R_+^*$'', we will focus on contractions in the next paragraphs (except in the statements \ref{t:Sz} through \ref{c:szek} which we kept as general as possible), but everything works for expansions as well. Of course, $\R_+$ can be replaced by any semi-open interval. 

One can obtain a contraction by taking the time-$1$ map of a \emph{smooth} \emph{contracting} vector field on $\R_+$, that is a vector field vanishing only at $0$ and ``pointing leftwards'' everywhere else, \emph{i.e.}\ of the form $u\partial_x$ (where $x$ is the coordinate on $\R_+$) with $u:\R_+\to\R_-$ vanishing only at $0$ (we will often identify the vector field with the corresponding function $u$). 

And one actually has the following \emph{partial} converse:

\begin{theoreme}[Szekeres \cite{Sz}, Sergeraert \cite{Se}, Yoccoz \cite{Yo2}] 
\label{t:Sz} 
Let $k\in\N$, $k \ge 2$, and let $f$ be a $C^k$-diffeomorphism of $\R_+$ without fixed points in $\R_+^*$. Then $f$ is the time-$1$ map of the flow of a complete vector field of class $C^1$ on $\R_+$ and $C^{k - 1}$ on $\R_+^*$.
\end{theoreme}

One cannot hope for more than $C^1$ regularity on $\R_+$ in general in the above statement, as Sergeraert shows in \cite{Se} by exhibiting a $C^\infty$ contraction $f$ which does not imbed in any $C^2$ flow (cf.\ Section \ref{sss:sergeraert} for an outline of his construction). This fact is of importance to us because of the following:

\begin{theoreme}[``Kopell's Lemma'' \cite{Ko}] 
\label{t:Kopell}
Let $f$ and $g$ be two commuting diffeomorphisms of $\R_+$ of class
$C^2$ and $C^1$ respectively. If $f$ has no fixed point in $\R_+^*$ and $g$ has one, then $g = \id$. 
\end{theoreme}

\begin{corollaire}[cf.\ for example \cite{Na}] 
\label{c:szek}
Let $k\in\N$, $k \ge 2$, and let $f$ be a $C^k$-diffeomorphism of $\R_+$ without fixed points in $\R_+^*$. Then $f$ is the time-$1$ map of the flow of a \emph{unique} $C^1$ vector field on $\R_+$, which we call the \emph{Szekeres vector field of $f$}. This vector field is $C^{k - 1}$ on $\R_+^*$, and the $C^1$-centralizer of $f$ coincides with its flow. 
\end{corollaire}

In particular, for a $C^\infty$ contraction $f$ of $\R_+$, the $C^1$-centralizer of $f$, \ie the set of $C^1$ diffeomorphisms of $\R_+$ commuting with $f$, consists in a one-parameter group of $C^1$-diffeomorphisms which are actually $C^\infty$ when restricted to the open half-line. Now the $C^\infty$-centralizer consists precisely of those flow maps of the Szekeres vector field $\xi$ which are smooth \emph{on all of} $\R_+$. Let us denote by $\cS_{\xi}$ the subgroup of $\R$ made of the times $t$ for which the time-$t$ map of $\xi$ is smooth. This subgroup contains $\Z$ in the present situation since the time-$1$ map $f$ is assumed smooth. It can be all of $\R$, when $\xi$ itself is smooth, in which case the centralizer of $f$ is a one-parameter group of diffeomorphisms (in particular path-connected). But it can also be reduced to $\Z$, in which case the centralizer of $f$ is infinite cyclic, generated by $f$, which is the case in Sergeraert's construction mentioned above. \medskip

In order to study $\Z^n$-actions on one-dimensional manifolds and their possible deformations, it is important to know whether $\cS_{\xi}$ can be neither connected nor infinite cyclic (cf.\ \cite{B-EB}). The author answered this question in \cite{Ey}, combining Sergeraert's construction with Anosov--Katok-like methods of deformation by conjugation (introduced in \cite{AK}; see also \cite{FK} and the references therein) to construct a contracting vector field whose time-$1$ and time-$\alpha$ maps are smooth, for some irrational number $\alpha$, but whose time-$\frac12$ map is not $C^2$. Hence, the set of smooth times is dense in $\R$ (it actually contains a Cantor set), but is not all of $\R$. \medskip

In the construction of \cite{Ey}, the very good approximation of $\alpha$ by rational numbers played a crucial role. To make this statement more precise, let us recall the famous partition of $\R\setminus\Q$ into Diophantine and Liouville numbers. A real number $\alpha$ is said to satisfy a Diophantine condition of order $\nu>0$ if there exists a constant $C>0$ such that for every $(p,q) \in \Z \times \N^*$, $|q\alpha - p| > \frac C {q^{1 + \nu}}$, or in other words such that for every $q\in\N^*$, $\|q\alpha\|>\frac C {q^{1 + \nu}}$, where $\|q\alpha\|$ denotes the distance between $q\alpha$ and $\Z$. A \emph{Diophantine number} is a number satisfying such a condition for some $\nu>0$ (in particular, such a number is necessarily irrational), and a \emph{Liouville number} is an irrational number which is not Diophantine. Roughly speaking, Diophantine and Liouville numbers are respectively ``badly'' and ``well'' approximated by rational numbers. \medskip

For the construction of \cite{Ey} to work, $\alpha$ needed to be Liouville, and the author proved shortly afterwards in the (unsubmitted) preprint \cite{Ey2} that one could actually make the construction work \emph{for any} Liouville number $\alpha$ (cf.\ Theorem \ref{t:liouville}, case $d=1$). It was then natural to wonder whether, conversely, the presence, along with $1$, of a Diophantine number $\alpha$ in the set of smooth times of a $C^1$ contracting vector field would force the latter to be $C^\infty$ itself (cf.\ below for ``evidence'' pointing in this direction). As it turns out, the ``contracting'' hypothesis plays no role at that point, and the main new achievement of this paper is to give a positive answer to the last question for \emph{any} $C^1$ vector field on any interval (cf.\ Theorem \ref{t:principal}, case $d=1$). 

These two statements (for $\alpha$ Liouville and $\alpha$ Diophantine respectively) correspond to the case ``$d=1$'' in Theorems \ref{t:liouville} and \ref{t:principal} below, which extend them to \emph{families} of irrational numbers, for which we have the following ``family-version'' of the dichotomy Diophantine/Liouville: we say that some numbers $\alpha_1,\dots,\alpha_d$, with $d\in\N^*$, are \emph{simultaneously Diophantine} if there exist $\nu>0$ and $C>0$ such that for every $q\in\N^*$, $\max(\|q\alpha_1\|,\dots,\|q\alpha_d\|)>\frac C {q^{1 + \nu}}$ (this requires one of these numbers, at least, to be irrational). In particular, for a single number, \emph{simultaneously Diophantine} just means \emph{Diophantine}, but a pair of simultaneously Diophantine numbers may consist of two (individually) Liouville numbers. 

\begin{theoreml}
\label{t:liouville} 
Let $d\in\N^*$. For any family $(\alpha_1,\dots,\alpha_d)$ of \emph{non} simultaneously Diophantine irrational numbers, there exists a complete $C^1$ contracting vector field on $\R_+$ whose time-$t$ map is $C^\infty$ for every $t\in\Z+\alpha_1\Z+\dots+\alpha_d\Z$ but not $C^2$ for some other times $t$ (in particular, $\xi$ is not $C^2$). 
\end{theoreml}

\begin{remarque} By construction (but we will see that it is in fact unavoidable), the smooth time-$t$ maps of the corresponding vector field will be infinitely tangent to the identity at $0$, the vector field will be smooth on $\R_+^*$ and its derivative will vanish at $0$. It is then easy to glue such vector fields together to prove that the above statement remains true if one replaces ``contracting vector field on $\R_+$'' by ``vector field on any interval $I$ of $\R$''.
\end{remarque}

\begin{theoremp} 
\label{t:principal} 
Let $d\in\N^*$ and let $(\alpha_1,\dots,\alpha_d)$ be a family of simultaneously Diophantine numbers. Given a complete $C^1$ vector field $\xi$ on an interval $I$ of $\R$, if the time-$t$ map of $\xi$ is $C^\infty$ for every $t\in\Z+\alpha_1\Z+\dots+\alpha_d\Z$, then $\xi$ itself is $C^\infty$. 
\end{theoremp}

Furthermore, the central general estimates involved in the proof of this statement also imply that, regardless of arithmetic considerations, if the set of smooth times contains $1$ and an irrational number, then it \emph{must} have the cardinality of the continuum:  

\begin{theoremc} 
\label{t:cantor} 
Let $\xi$ be a complete $C^1$ vector field on some interval $I$ of $\R$. If the set of $t\in\R$ such that the time-$t$ map of $\xi$ is $C^\infty$ contains $1$ and some irrational number $\alpha$, then it actually contains a Cantor set $K_\alpha$ (depending only on $\alpha$, not on $\xi$).
\end{theoremc}

\begin{remarque}
\label{r:Q} It is possible, however, to construct a $C^1$ contracting vector field whose set of smooth times is dense in $\R$ and countable, equal to $\Q$ for example. This result is part of an unpublished work with C. Bonatti where we also construct, for any given $r\ge2$ and any irrational number $\alpha$, a $C^1$ contracting vector field whose set of $C^r$ times is reduced to $\Z+\alpha\Z$. 
\end{remarque}

The complementary statements of Theorems \ref{t:liouville} and \ref{t:principal} look ``parallel'' to well-known results about the linearization of circle diffeomorphisms. We will see that Theorem \ref{t:cantor}, on the other hand, contrasts with some other result on the same topic. We now explicit the relation between the two situations and, in the process, extend Theorems \ref{t:principal} and \ref{t:cantor} to $C^0$ flows. 

\subsection{Relation with circle diffeomorphisms}
\label{ss:cercle}

In this article, ``the circle'' refers to $\T^1=\R/\Z$. Given $\alpha\in\R$, we denote by $\bar \alpha$ its projection to $\T^1$. Furthermore, we denote by $\Homeot$ the set of homeomorphisms of $\R$ commuting with the unit translation. Recall that, given $f\in \Homeot$, for every $x\in \R$, $(\frac{f^n(x)-x}{n})_{n\in\N^*}$ converges towards a number which does not depend on $x$, called the \emph{translation number of $f$} and denoted by $\tau(f)$. Note that $f$ has a fixed point if and only if $\tau(f)=0$. Now if $g$ is a homeomorphism of the circle and if $\tilde g\in\Homeot$ is a lift of $g$ to $\R$, $\overline{\tau(\tilde g)}$ depends only on $g$ (not on the lift) and is called the \emph{rotation number} of $g$, denoted by $\rho(g)$.

Here are now the famous linearization results alluded to in the previous paragraph:

\begin{theoreme}[Herman \cite{He}, Fayad-Khanin \cite{Fa-Kh}]
\label{t:H}
Let $d\in\N^*$. For any family $(\alpha_1,\dots,\alpha_d)$ of non simultaneously Diophantine irrational numbers, there exist pairwise commuting $C^\infty$-diffeomorphisms of the circle $g_1,\dots,g_d$ with respective rotation numbers $\bar \alpha_1,\dots,\bar\alpha_d$ which are not $C^1$-conjugated to the corresponding rotations.
 \end{theoreme}

\begin{theoreme}[\cite{He,Yo,Fa-Kh}]
\label{t:HY}
Given $d\in\N^*$, let $g_1,\dots,g_d$ be smooth pairwise commuting diffeomorphisms of the circle. If their rotation numbers are simultaneously Diophantine, then they are simultaneously smoothly conjugated to the corresponding rotations. 
\end{theoreme}

A first natural reflex is to wonder whether Theorems \ref{t:liouville} and \ref{t:principal} can somehow be deduced from the above statements. Can one use the ``exotic'' circle diffeomorphisms of Theorem \ref{t:H} to construct an ``exotic'' flow as in Theorem \ref{t:liouville}? Can one ``translate'' the flow of Theorem \ref{t:principal} in terms of circle diffeomorphisms to which Theorem \ref{t:HY} applies and get the desired conclusion? We will see that the answer is \emph{no}, though there does exist a correspondence between circle diffeomorphisms and flows, but not any flow: flows which define a \emph{free} action of the group $(\R,+)$ on the real line $\R$ (or any open interval, cf.\ below). We, on the other hand, are interested primarily in \emph{non-free} $\R$-actions (contractions, in particular, fix~$0$). We will see that there are fundamental differences between the two situations (free and non-free), and the challenge faced in the present work is precisely to figure out which techniques used in the free case can be adapted to the non-free one, and how. \medskip
 
But let us first explicit the aforementioned correspondence. Basically, to a circle diffeomorphism with irrational rotation number $\alpha$ corresponds a free $\R$-action on $\R$ with smooth time-$1$ and time-$\alpha$ maps, and conversely. More precisely, the correspondence between both settings in the general case ($d\in\N^*$) is as follows, with $(\alpha_1,\dots,\alpha_d)$ $\in\R^d\setminus\Q^d$ (without loss of generality, we assume $\alpha_1\in\R\setminus\Q$), and $G:=\Z+\alpha_1\Z+\dots+\alpha_d\Z$.

First, consider smooth commuting diffeomorphisms $g_1,\dots,g_d$ of $\R/\Z$ with respective rotation numbers $\bar \alpha_1,\dots,\bar \alpha_d$. Since $\alpha_1$ is irrational, Denjoy's theorem claims that $g_1$ is conjugated by a homeomorphism $\f$ to the rotation $R_{\bar \alpha_1}$. Since $g_2,\dots,g_d$ commute with $g_1$, their conjugates by $\varphi$ commute with $R_{\bar \alpha_1}$, so are themselves rotations (by irrationality of $\alpha_1$). And since their rotation numbers are still $\bar \alpha_2,\dots,\bar \alpha_d$, this actually means that $\varphi$ simultaneously conjugates $g_1,\dots,g_d$ to $R_{\bar \alpha_1},\dots,R_{\bar \alpha_d}$. Hence, $g_1,\dots,g_d$ imbed in the one-parameter group $(\varphi\circ R_{\bar t}\circ \varphi^{-1})_{t\in\R}$ of circle homeomorphisms. The latter lifts uniquely to a one-parameter subgroup (or $C^0$ flow) $(f^t)_{t\in\R}$ of $\Homeot$ conjugated to the translation flow $(T_t)_{t\in\R}$ by a lift $\tilde\f$ of $\f$, in which $f^1=T_1$ is smooth, as well as $f^{\alpha_i}$ for every $i$, as a lift of $g_i$. This flow, like the translation flow, acts freely on $\R$.

Conversely, let $(f^t)_{t\in\R}$ be a $C^0$ flow on $\R$ defining a \emph{free} continuous $\R$-action on $\R$ and such that $f^t$ is $C^\infty$ for every $t\in G$. Up to a smooth conjugacy, we can assume $f^1 = T_1$. Then, for every $t\in\R$, $f^t$ is a homeomorphism of $\R$ commuting with the unit translation, so it has a well-defined translation number $\tau(f^t)$, and this translation number is $t$. (Indeed, $\tau$ is a group morphism when restricted to abelian subgroups of $\Homeot$, and in particular, $t\mapsto \tau(f^t)$ is a continuous homomorphism from $(\R,+)$ to itself sending $1$ to $1$, hence the identity.) Hence, $f^{\alpha_1},\dots,f^{\alpha_d}$ induce commuting $C^\infty$-diffeomorphisms $g_1,\dots,g_d$ of the circle, whose rotation numbers are $\bar\alpha_1,\dots,\bar\alpha_d$. 

With this in mind, we claim that Theorems \ref{t:H} and \ref{t:HY} are respectively equivalent to the following statements: 
\emph{
\begin{itemize}
\item if $\alpha_1,\dots,\alpha_d$ are non simultaneously Diophantine, there exists a $C^0$ flow $(f^t)_{t\in\R}$ \textbf{acting freely} on $\R$ whose time-$t$ map is smooth for every $t\in \Z+\alpha_1\Z+\dots+\alpha_d\Z$ but which is not $C^1$-conjugated to the translation flow $(T_t)_{t\in\R}$;
\item if $\alpha_1,\dots,\alpha_d$ are simultaneously Diophantine, every $C^0$ flow $(f^t)_{t\in\R}$ \textbf{acting freely} on $\R$ whose time-$t$ map is smooth for every $t\in \Z+\alpha_1\Z+\dots+\alpha_d\Z$ is smoothly conjugated to the translation flow.
\end{itemize}
}

Note that in the second statement, since the action is free, ``is smoothly conjugated to the translation flow'' (which is the flow of the unit vector field on $\R$) can be replaced by ``is the flow of a smooth vector field'' (necessarily non-vanishing). 

Now we claim that one can use Theorem \ref{t:principal} to obtain the very same (reformulated) second statement for \emph{non-free} actions. This is the content of the following theorem, \textbf{whose statement, in conclusion, holds in full generality} (for free and non-free actions). 

\begin{theoremr} \label{t:reform} 
Let $d\in\N^*$ and let $(\alpha_1,\dots,\alpha_d)$ be a family of simultaneously Diophantine numbers. If $(f^t)_{t \in \R}$ is a one-parameter group of 
 homeomorphisms of some interval $I$ of $\R$ \textbf{acting non-freely} on $I$ and such that $f^t $ belongs to $\Diff^\infty_+(I)$ for every $t\in\Z+\alpha_1\Z+\dots+\alpha_d\Z$, then $(f^t)_{t \in \R}$ is the flow of a $C^\infty$ vector field on $I$. 
\end{theoremr}

More precisely, as we will see in Section \ref{ss:reduction}, it follows from the aforementioned works of Szekeres, Kopell and Sergeraert that, \emph{without need of any arithmetic condition on $(\alpha_1,\dots,\alpha_d)$ $\in\R^d\setminus\Q^d$}, a $C^0$ flow as above (acting \textbf{non-freely}) is \emph{automatically} the flow of a $C^1$ vector field on $I$ $(\star)$, which is even $C^\infty$ on the complement of its vanishing points $(\star \star)$. This shows, in particular, that there can be no strict analogue, in terms of regularity, of Theorem \ref{t:H} (or its reformulation) in the non-free case, where everything will automatically be $C^1$ according to $(\star)$. But more importantly, $(\star\star)$ says that in the non-free case, the question of the regularity of the action is concentrated at the fixed points; outside of them, where the action is free, everything is automatically smooth. Hence the \emph{results} about free actions (translating Theorems \ref{t:H} and \ref{t:HY}) can tell us nothing we do not already know regarding Theorem \ref{t:principal}, and cannot be used to construct the non-free actions of Theorem \ref{t:liouville} (even in an adapted version where $C^1$ would be replaced by $C^2$). 
 
However, our \emph{proofs} are deeply inspired from the circle ones, of which they borrow and adapt some techniques, but with some significant differences:
 
\paragraph{``Liouville case''.} As explained in more details in Section \ref{s:liouville}, the proof of Theorem \ref{t:liouville}, like that of Theorem \ref{t:H} (at least a possible one) uses an Anosov-Katok-like method of deformation by conjugation. But in the non-free case, some extra ingredient is needed. In a few words, the Anosov-Katok construction consists in starting with a smooth flow and applying to it successive carefully manufactured conjugations which, at the limit, leave some prescribed times of the flow smooth and some others not. In the circle/free case, the starting flow is simply the flow of rotations/translations. In the half-line/non-free case, there is no such ``obvious'' flow to start with. And actually, Sergeraert shows in \cite{Se} that in the half-line contracting case, to end up with a flow with \emph{one} (non-identity) smooth map and \emph{one} non-smooth one already requires a clever choice of starting flow (which is not the case on the circle, cf.\ Section \ref{sss:sergeraert}). Still, there are some degrees of freedom in Sergeraert's construction, and the challenge, for us, is to make sure that, given a  \emph{prescribed} \emph{dense} subset of smooth times as in Theorem \ref{t:liouville}, one can define a Sergeraert-type initial flow \emph{in terms of this data} in such a way that an Anosov-Katok-type process performed on this flow ``converges'', \ie leads to a flow with the desired properties. 

\paragraph{``Diophantine case''.} In the circle/free case, the proof of Theorem \ref{t:HY} can be divided in two parts: proving the $C^1$ conjugacy (which already requires the arithmetic condition), and then ``bootstrapping'' from $C^1$ to $C^\infty$. In the non-free case (Theorem \ref{t:reform}), as already mentioned (and proved in Section \ref{sss:reduction1}), the analogue of the first part (and actually more, cf.\ below) is taken care of by previous works, without need of the arithmetic condition, reducing Theorem~\ref{t:reform} to Theorem~\ref{t:principal}. 

So in principle, our task would be to go through the whole proof of the ``bootstrapping'' part (as well as the \emph{a priori} bounds involved therein) in \cite{He,Yo,Fa-Kh}, translate every step in terms of $\R$-actions, and see, each time, whether the freeness hypothesis can be removed (this requires in particular getting rid of any argument specific to the circle and any allusion to translations, which, again, are the canonical model for free $\R$-actions on $\R$, while there is no such model in the non-free case). More will be said on these ``steps'' in Section \ref{s:Diophantine}, and in particular in \ref{ss:conj-it} and \ref{ss:iteres}.

However, there turns out to be a substantial difference between the proof we give in Section~\ref{s:Diophantine} and the above scheme.  In a nutshell, this comes from the fact that in the circle case, the bootstrapping step starts with a flow which is $C^1$ conjugated to the translation flow, while here, we start with the flow of a $C^1$ vector field. Observe that in the free case, being generated by a $C^1$ (non-vanishing) vector field is equivalent to being $C^{2}$ conjugated to the translation flow (the conjugation $\phi$ being related to the generating vector field $u\partial_x$ by $\phi'=1/u$). This small shift in the regularity from which one starts, reverberated in the use we make of Hadamard convexity inequalities (cf.\ Section \ref{ss:norms}) turns out to have a big impact on the proof (cf.\ Remarks \ref{r:C1-C0} and \ref{r:C1-C0-2}). Namely, we reprove most analogues of the estimates of the bootstrap part of \cite{Yo} \emph{without need of the arithmetic condition}. A byproduct of this is Theorem \ref{t:cantor}, which, as Theorem \ref{t:principal}, admits an extension to \emph{non-free} $C^0$ actions, in particular actions on non-open intervals (cf.\ Section \ref{sss:reduction1} for a proper proof of the reduction):

\begin{theoremrb} \label{t:reform2} 
Let $(f^t)_{t \in \R}$ be a one-parameter group of 
 homeomorphisms of an interval $I$ of $\R$ \emph{acting non-freely on $I$}. If $\{t\in\R: f^t \in \Diff^\infty_+(I)\}$ contains $\Z + \alpha\Z$ for some irrational number $\alpha$, then it actually contains a Cantor set $K_\alpha$ (depending only on $\alpha$).
 \end{theoremrb}

On the other hand, the following theorem of Yoccoz shows that the analogue of the above statement is false in the \emph{free} case:

\begin{theoreme}[Yoccoz, \cite{Yo2} p.\ 207]
\label{t:Y}
There exists $g\in\Diff^\infty_+(\T^1)$ with an irrational rotation number such that the centralizer of $g$ in $\Diff^\infty_+(\T^1)$ is reduced to the infinite cyclic subgroup generated by $g$.
\end{theoreme}
 
Indeed, as seen before, to such a $g$, with rotation number $\alpha\mod1$, corresponds a one-parameter subgroup $(f^t)_{t \in \R}$ of $\Homeo_+(\R)$ among which $f^1=T_1$, $f^\alpha$ is a lift of $g$ and any smooth $f^t$ induces a smooth circle diffeomorphism commuting with $g$. Thus, the ``triviality'' of the centralizer implies that in this case, the set $\{t\in\R:f^t\in \Diff_+^\infty(\R)\}$ \emph{is exactly} $\Z+\alpha\Z$, unlike in Theorem \ref{t:reform2}. 
 \medskip

Let us conclude this introduction with two final comments about the above statements and their possible variants. 

\begin{remarque}
\label{r:finite}
For a single diffeomorphism, Theorem \ref{t:HY} has a more precise statement in finite (not necessarily integral) regularity, and so does Theorem \ref{t:principal}: if $\alpha$ is Diophantine of order $\nu$ and the time-$t$ maps of a $C^1$ vector field $\xi$, for $t\in\Z+\alpha\Z$, are of class $C^k$ with $k\ge 3$, then $\xi$ is a $C^\gamma$ pull-back of a $C^\gamma$ vector field for any $\gamma < k-1-\nu$. It will be the aim of another article to present this refined statement and study its optimality. 
\end{remarque}

\begin{remarque}
\label{r:manifold}
The statements of Theorems \ref{t:principal}, \ref{t:reform} (without freeness assumption) and \ref{t:cantor} remain true for $I=\T^1=\R/\Z$ since this case reduces to the case $I=\R$ by lifting. So in the end, these statements hold for any $1$-dimensional manifold, with or without boundary.
\end{remarque}

The outline of the article is now simple: in Section \ref{s:Diophantine}, we prove Theorems \ref{t:principal} and \ref{t:cantor} (and deduce Theorems \ref{t:reform} and \ref{t:reform2}), and in Section~\ref{s:liouville}, we prove Theorem \ref{t:liouville}. These two sections are completely independent.

\subsection{Formulae and notations}
\label{ss:notations}

\subsubsection{Derivation formulae}
\label{sss:deriv}

In this paragraph, $I$ denotes any interval of $\R$. If $\f:I\to\R$ is sufficiently regular, given $r\in\N$, for readability reasons, we will denote the $r$-th derivative of $\f$ by $D^r\f$ rather than $\f^{(r)}$. \medskip

In both parts of this article, we will make use of Fa\`a di Bruno's formula below, which can be easily checked by induction. It holds for any sufficiently regular functions $\varphi, \psi:I\to\R$. For $r\in \N^*$, $\Pi_r$ denotes the set of partitions of $[\![1, r]\!]:=[1,r]\cap\N$, $|\pi|$ the number of ``blocks'' in a partition $\pi \in \Pi_r$ and, for such a block $B$ of $\pi$ (which we abusively denote by ``$B \in \pi$''), $|B|$ denotes the number of elements of $B$. Then, 
\begin{equation}
\label{e:Faa}
D^r(\varphi \circ \psi) = \sum_{\pi \in \Pi_r} \left( \left(D^{|\pi|} \varphi\right) \circ \psi \prod_{B \in \pi} D^{|B|} \psi \right).
\tag{Fa\`a}
\end{equation}
Gathering the terms corresponding to a same value of $|\pi|$, this can be rewritten as:
\begin{equation}
\label{e:Faa2}
D^r(\varphi \circ \psi) = \sum_{k=1}^r (D^k\f)\circ \psi \times B_{r,k}(D\psi,\dots,D^{r-k+1}\psi) 
\tag{Fa\`a'}
\end{equation}
where $B_{r,k}$ is a so-called \emph{Bell polynomial}, which is a polynomial in $m=r-k+1$ variables $X_1$,\dots, $X_m$ in which every monomial is of the form $c X_1^{j_1}\dots X_{m}^{j_m}$ with $c\in\R$, $j_1+\dots+j_m=k$ and $j_1+2j_2+\dots+mj_m=r$, meaning $B_{r,k}$ is homogeneous of degree $r$ when $X_i$ is given weight~$i$.\medskip

If $g$ is now a $C^\infty$ diffeomorphism of $I$, then starting with the equality
$(Dg^{-1} \circ g) Dg = 1$ and using Fa\`a di Bruno's Formula, one gets by induction on $r$:
\begin{equation}
\label{e:inverse}
\tag{Inv}
(D^r g^{-1}) \circ g= \frac {P_r(Dg,...,D^rg)}{(Dg)^{2r+1}},
\end{equation}
for some universal polynomial $P_r$ in $r$ variables without constant term.\medskip

Now if $g$ and $h$ are smooth orientation preserving diffeomorphisms of $I$ (so that the logarithm of their first derivative is well-defined), it will prove fruitful to use the usual ``chain-rule'' under the form:
\begin{equation}
\label{e:D}
\log D(g \circ h) = (\log Dg) \circ h + \log Dh,
\tag{Ch}
\end{equation}
so that the nonlinearity differential operator defined by
$$Ng = D \log Dg = \frac{D^2 g}{Dg}$$
satisfies:
\begin{equation}
N(g \circ h) = ((Ng) \circ h ) Dh + Nh,
\tag{N}
\end{equation}
and by induction:
\begin{equation}
\label{e:E}
Ng^n = \sum_{i = 0}^{n- 1} ((Ng) \circ g^i) Dg^i.
\tag{N'}
\end{equation}
For higher order derivatives, one has, for every $r \in \N^*$ (the formulas and their labeling are those of \cite{Yo}):
\begin{align*}
\label{e:G}
D^r( \log D(g \circ h) ) = D^{r - 1}(N(g\circ h))= 
&\left((D^{r - 1} Ng) \circ h\right)(Dh)^r + D^{r - 1}Nh \\
& + \sum_{l = 1}^{r - 1} \left( (D^{r - l- 1} Ng) \circ h \right) (Dh)^{r - l} \,
Q_l^r(Nh, \dots ,D^{l- 1}Nh)
\tag{$\mathrm{G}$}
\end{align*}
and
\begin{equation}
\label{e:H}
D^r(\log Dg^n) = \sum_{l = 0}^{r-1}
 \sum_{i = 0}^{n- 1} \left( (D^{r - l} \log Dg) \circ g^i \right) 
(Dg^i)^{r - l} \, R_l^r(Ng^i, \dots, D^{l- 1}Ng^i)
\tag{$\mathrm{H}$}
\end{equation}
where $Q_l^r$ and $R_l^r$ (as well as $A_l$ and $B_l$ below) denote universal polynomials of $l$ variables $X_1, \dots,X_l$, homogeneous of weight $l$ if $X_i$ is given weight $i$, and equal to $1$ if $l = 0$.

Finally, one has the following relations between the regular derivatives and those of the nonlinearity:
\begin{equation}
\label{e:A}
\forall l \in \N, \quad D^{l + 1}g = A_l(Ng, \dots, D^{l- 1}Ng) Dg,
\tag{A} 
\end{equation}
\begin{equation}
\label{e:B}
\forall l \in \N^*, \quad D^{l - 1} Ng = B_l \left( \frac{D^2g}{Dg}, \dots, \frac{D^{l + 1}g}{Dg} \right).
\tag{B} 
\end{equation}

\subsubsection{``Norms''}
\label{sss:norms}

Given a map $\f:I\to\R$, we write
$$\|\f\|_0 = \sup_I|g|\quad\in\R_+\cup\{+\infty\}$$
(the interval $I$ on which the supremum is taken will usually be clear from the context; namely, it will be $\R_+$ in Section \ref{s:liouville} and $\R$ in Section \ref{s:Diophantine}. If not, we will specify it by writing $\|\f\|_{0,I}$). If $\f$ is $C^r$, for $r\in\N$, we write
$$\|\f\|_r=\max_{0\le j\le r}\|D^j\f\|_0\quad \in\R_+\cup\{+\infty\}$$
(this is finite if $I=\R$ and $\f$ is periodic, which will be the case in Section \ref{s:Diophantine}, cf.\ Section \ref{ss:norms}).

\section{Proofs of Theorems \ref{t:reform} and \ref{t:reform2}}
\label{s:Diophantine}
 
In Section \ref{ss:reduction}, we reduce Theorems \ref{t:reform} and \ref{t:reform2} to a particular case of Theorems \ref{t:principal} and \ref{t:cantor}, namely that of a $1$-periodic (possibly vanishing) vector field on $\R$. Next, in Section \ref{ss:norms}, we present general facts about $C^r$-norms and composed maps that will be used throughout the subsequent sections. The strategy of the proof of Theorem \ref{t:principal} (in the reduced setting) is then explained in Sections \ref{ss:conj-it} and \ref{ss:iteres}. First, in Section \ref{ss:conj-it}, we show that the regularity of the vector field $\xi$ can be deduced from a uniform control on the derivatives of $f^t$ for $t$ in some subset of $\Z+\alpha_1\Z+\dots+\alpha_d\Z$ dense in $[0,1]$ (cf.\ Proposition \ref{p:ceq} and Theorem \ref{t:T}). This control is obtained in Section \ref{ss:iteres} by composition, using the arithmetic condition, from the general estimates of Lemma \ref{l:16} (about $f^t$ for some very specific values of $t$) which do not require this condition. This central lemma is proved in Section \ref{ss:gen}. Intermediate (general) estimates leading to it are then used in Section \ref{ss:cantor} to prove Theorem \ref{t:cantor}.
 
This strategy and the computations it involves are very much inspired from \cite{Yo} (itself building on \cite{He}) for a single irrational number, and from \cite{Fa-Kh} for the case of families (see the introduction for the parallel between our situation and diffeomorphisms of the circle). This is acknowledged, throughout the text, by explicit references to the original statements of which ours are inspired, when applicable. No deep knowledge of these articles is assumed, though; our proofs are mostly self-contained. As a matter of fact, there is a substantial difference between the two settings. As already mentioned in the introduction and further explained in Remarks \ref{r:C1-C0} and \ref{r:C1-C0-2}, the fact that we start with the flow of a $C^1$ vector field while the bootstrapping in \cite{Yo,Fa-Kh} starts with a $C^1$ conjugacy to a smooth flow allows us to prove the central estimates (those of Lemma \ref{l:16} and the intermediate ones leading to it) without need of the Diophantine condition, unlike \cite{Yo,Fa-Kh}. In fact, in our case, some of these estimates can be deduced directly from much more general statements (about control of norms of composed maps), while this is not quite the case in \cite{Yo,Fa-Kh} (cf.\ for example Lemmas \ref{l:short-it} and \ref{l:prem-it} and Remark \ref{r:C1-C0-2} about them). We have tried to emphasize this difference by extracting from the proofs of \cite{Yo} what could be made into general statements in our situation, statements which we have isolated in a dedicated Section \ref{ss:norms}, keeping for Sections \ref{ss:conj-it}, \ref{ss:iteres} and \ref{ss:gen} only what is specific to our flow with rationally independent smooth times. 

Furthermore, like in \cite{Fa-Kh}, the fact that we start with some $C^\infty$ flow maps rather than $C^k$ ones with $k$ finite, and that we are not trying to ``optimize our use of derivatives'' also simplifies the bootstrapping procedure compared to \cite{Yo}.

\subsection{Reduction}
\label{ss:reduction}
   
\subsubsection{From Theorems \ref{t:reform} and \ref{t:reform2} to Theorems~\ref{t:principal} and \ref{t:cantor}}
\label{sss:reduction1}
  
This reduction follows directly from Proposition \ref{p:nonfree} below, which mainly relies on the classical Theorems \ref{t:Sz} and \ref{t:Kopell} recalled in the introduction, as well as the following one:

\begin{theoreme}[Takens \cite{Ta}] 
\label{t:takens}
If a $C^\infty$-diffeomorphism $f$ of some semi-open interval $[a,b)$ has no fixed point in $(a,b)$ and is not infinitely tangent to the identity at $a$, its Szekeres vector field is $C^\infty$ on $[a, b)$.

Furthermore, if a $C^\infty$-diffeomorphism $f$ of some open interval $(a,b)$ has a unique fixed point $c\in (a,b)$ where it is not infinitely tangent to the identity, then its Szekeres vector fields on $(a,c]$ and $[c, b)$ match up smoothly at $c$.
\end{theoreme}

For a smooth diffeomorphism $f$ of an interval $I$, we denote by $\ITI(f)$ the set of points of $I$ where $f$ is infinitely tangent to the identity (ITI), \ie all derivatives (including the ``$0$-th'') of $f-\id$ vanish.

\begin{proposition}\label{p:nonfree} 
Let $(f^t)_{t \in \R}$ be a one-parameter group of homeomorphisms of an interval $I$ of $\R$, such that $f^t$ belongs to $\Diff^\infty_+(I)$ for every $t $ in $\Z + \alpha\Z$ for some irrational number $\alpha$, and $f^1$ has at least one fixed point. Then $(f^t)_{t \in \R}$ is the flow of a (unique) $C^1$ vector field on $I$, which is in addition $C^\infty$ on the complement of $\ITI(f^1)$.
\end{proposition}

\begin{remarque}
\label{r:free}
The non-freeness of an action of $(\R,+)$ on some interval $I$ defined by a one-parameter subgroup $(f^t)_{t\in\R}$ of $\Homeo_+(I)$ is equivalent to $f^1$ having a fixed point. More generally, for every $t\in\R^*$, $\Fix(f^t)=\Fix(f^1)$, or equivalently: for any $x\in I$, the stabilizer $\St(x)$ of $x$ under the action $(t,x)\mapsto f^t(x)$ is either $\{0\}$ (meaning $x$ is fixed only by the identity) or $\R$ (meaning $x$ is fixed by the whole one-parameter group). 

This comes from the order on $I$. Indeed, $\St(x)$ is a closed subgroup of $\R$ (so either $\{0\}$, infinite cyclic or $\R$) which is, in addition, invariant under multiplication by $2$ (so only $\{0\}$ and $\R$ remain), meaning that, for every $t\in\R$, $f^t(x)=x$ \emph{if and only if} $f^{2t}(x)=x$. Indeed, one implication is straightforward (if $f^t(x)=x$, then $f^{2t}(x)=x$). Conversely, if $f^t(x)\neq x$, letting $(a,b)$ be the connected component of $x$ in $I\setminus \Fix(f^t)$, the sequence $((f^t)^k(x))_{k\in\N}$ is (strictly) monotonous (with $a$ or $b$ as a limit) and in particular, $f^{2t}(x)\neq x$. \end{remarque}

\begin{proof}[Proof of Proposition \ref{p:nonfree}]
Assume the action is not trivial (which, according to Remark \ref{r:free}, is equivalent to $f=f^1$ not being the identity). Let $(a,b)$ be a connected component of $I\setminus \Fix(f^1)$. Since $f^1$ has at least one fixed point, at least one of the endpoints of $(a,b)$, say $a$, is finite and belongs to $I$. So, as a smooth contraction or \emph{expansion} on $[a,b)$, $f$ has a well-defined Szekeres vector field $\xi_a$ on $[a,b)$. If $b$ also belongs to $I$, \ie is a fixed point of $f$, $f$ also has a Szekeres vector field $\xi_b$ on $(a,b]$. Usually, the two do not necessarily coincide (cf.\ \cite{Ko}). But in our present situation, they do. 

Indeed, $f^\alpha$, which has the same fixed points as $f$, induces a smooth diffeomorphism of $[a,b]$ commuting with $f$. Thus, by Corollary \ref{c:szek}, $f^\alpha$ coincides there with some time-$\beta$ and time-$\gamma$ maps of $\xi_a$ and $\xi_b$ respectively, and one actually has $\alpha=\beta=\gamma$. Indeed, if $(\phi_a^t)_t$ (resp.\ $(\phi_b^t)_t$) denotes the flow of $\xi_a$ (resp.\ $\xi_b$), for any given $x_0\in (a,b)$, 
$\psi_a:t\mapsto \phi_a^t(x_0)$, $\psi_b:t\mapsto \phi_b^t(x_0)$ and $\psi:t\mapsto f^t(x_0)$ define three homeomorphisms from $\R$ onto $(a,b)$ which conjugate $f$ to $T_1$ and $f^\alpha$ to $T_\beta$, $T_\gamma$ and $T_\alpha$ respectively. These three translations are thus conjugated to one another \emph{by homeomorphisms commuting with $T_1$} ($\psi_b^{-1}\psi_a$ and $\psi_a^{-1}\psi$), so they must have the same translation number, which gives the desired equality. 

Hence, on $(a,b)$, $\phi_a^t=\phi_b^t$ for $t=1$ and $\alpha$, so for every $t\in\Z+\alpha\Z$ which is dense in $\R$ since $\alpha$ is assumed irrational, so, by continuity, for every $t\in\R$. This means $\xi_a=\xi_b$ on $(a,b)$. Thus, it makes sense to talk about \emph{the} Szekeres vector field of $f$ on each connected component of $I\setminus\Fix(f)$ (where it is smooth), and to extend it to $I$ by $0$. Let $\xi$ denote the resulting (continuous) vector field on $I$.

If $c$ is a fixed point of $f$ where $f$ is not ITI, $c$ is an isolated fixed point and by Takens' result~\ref{t:takens}, $\xi$ is $C^\infty$ near $c$. What is left to prove is that $\xi$ is $C^1$ near the ITI fixed points and that $(f^t)_{t\in\R}$ is its flow. If $f$ had only isolated ITI fixed points, for such a point $c$, the regularity of $\xi$ near $c$ would directly follow from the $C^1$ regularity of the Szekeres vector field of a contraction (or expansion) of a semi-open interval, and the observation that, if the contraction is ITI at $c$, its Szekeres vector field is necessarily $C^1$-flat at $c$. To settle the case of non-isolated ITI fixed points, one applies a theorem by Yoccoz \cite[Chap.\ 4, Theorem 2.5]{Yo2} which claims the ``continuous dependence'' (in $C^1$-topology) of the Szekeres vector field with respect to its time-$1$ map (in $C^2$-topology) and shows in particular that if the time-$1$ map is $C^2$-close to the identity, the Szekeres vector field is $C^1$-small. 

Hence $\xi$ is $C^1$ and its time-$1$ map, $f^1$, is well-defined on $I$, so $\xi$ is complete. We already know that $f^\alpha$ coincides on $I\setminus \Fix(f)$ with the time-$\alpha$ map $\phi_\xi^\alpha$ of $\xi$, and this is also true on $\Fix(f)$ since $\xi$ vanishes and $f^\alpha$ is the identity there. Thus $f^t=\phi_\xi^t$ for all $t\in\Z+\alpha\Z$, and once again we conclude by continuity and density.
 \end{proof}
 
\subsubsection{Reduction of Theorems~\ref{t:principal} and \ref{t:cantor} to the periodic case}
 \label{sss:reduction1}

We are thus reduced to proving Theorems~\ref{t:principal} and \ref{t:cantor}. We claim that it is sufficient to prove them in the case where $\alpha_1,\dots,\alpha_d$ (resp.\ $\alpha$) belong to $[0,\frac14)$, $I=\R$ and $\xi$ is $1$-periodic (so in a way, we are back on the circle, but we allow global fixed points, unlike in the case of irrational rotations). For readability reasons, we explain this reduction only for the case $d=1$ in Theorem \ref{t:principal}, but the general case is identical. 

So let $\xi$ be a complete $C^1$ vector field on some interval $I$, whose time-$t$ map is $C^\infty$ for every $t\in\Z+\alpha\Z$ for some irrational $\alpha$. First, replacing $\alpha$ by some $q\alpha-p$, $(p,q)\in\Z\times(\Z\setminus \{0\})$ if necessary, one may indeed assume that $\alpha$ belongs to $[0, \frac 1 4)$.

As we already saw, in Theorem \ref{t:principal}, the case of a non-vanishing $\xi$ (which implies that $I$ is open, otherwise, by completeness, $\xi$ would have to vanish on the boundary) is a consequence of Theorem \ref{t:HY}. This case will nevertheless be included in our proof. By smooth conjugation, we can assume, in this case, that $I$ is $\R$ and that the time-$1$ map $f^1$ is the unit translation $T_1$. In particular, $\xi$ is $1$-periodic.

For a vector field \emph{with} singularities, we now explain how to reduce again to a $1$-periodic vector field on $\R$. Assume $\xi$ vanishes somewhere. Applying Proposition \ref{p:nonfree} to its flow, we get that $\xi$ is actually smooth on the complement of $\ITI(f^1)$. Assume $\ITI(f^1)$ is nonempty (otherwise we are done) and let $a$ be one of its elements. For Theorem \ref{t:principal}, assuming $\alpha$ is Diophantine, it is sufficient to prove that $\xi$ is smooth on $I^+=I\cap [a,+\infty)$ and $I^-=I\cap ]-\infty,a]$: this is clear if one of these intervals is a singleton, and otherwise, if $\xi$ is smooth on both, since $f^1$ is ITI at $a$, $\xi$ will necessarily be infinitely flat at $a$ on both sides\footnote{One can check, without hypothesis of regularity on $\xi$, that, at any point of $\ITI(f^1)$, $\xi(x)\sim f(x)-x$, cf.\ Lemma 2.9 in \cite{Se} and Lemma 2.19 in \cite{B-EB}.} and thus smooth on the union $I$. Now in order to prove the smoothness on $I^+$, say, it is sufficient to prove it on any segment $[a,b]\subset I^+$, with $b>a$. Pick any such $b$. If $b\notin\ITI(f^1)$,  $\xi$ is smooth on a neighbourhood of $b$ in $[a,b]$, so we can consider a $C^1$ vector field $\zeta$ on $[a,b]$ which coincides with $\xi$ outside such a neighbourhood, is smooth on this neighbourhood and infinitely flat at $b$. Its time-$1$ and time-$\alpha$ maps are still smooth and now ITI at $b$, and $\xi$ is smooth on $[a,b]$ if and only if $\zeta$ is. So we are reduced to proving Theorem \ref{t:principal} in the case $I=[a,b]$ with $f^1$ ITI at $a$ and $b$, and, up to a smooth conjugacy, we can assume $a=0$ and $b=1$. The vector field under scrutiny, which is $C^1$-flat at the boundary, can then be extended to $\R$ as a $1$-periodic $C^1$ vector field, whose time-$1$ and time-$\alpha$ maps are $C^\infty$ since they are $C^\infty$ on $\R\setminus \Z$ and their restriction to every $[n,n+1]$, $n\in\Z$, is ITI at the boundary. This completes the reduction for Theorem \ref{t:principal}, and a similar argument reduces the proof of Theorem \ref{t:cantor} to the same particular case where $\xi$ is $1$-periodic on $\R$.\medskip

From now on, we will say we are in the \emph{reduced setting} of Theorems \ref{t:principal} and \ref{t:cantor} if, in addition to their respective hypotheses, one adds: $I=\R$, $\xi$ is $1$-periodic and $\alpha_1,\dots,\alpha_d$ (resp.\ $\alpha$) belong to $[0,\frac14)$.

\subsection{$C^r$-norms}
\label{ss:norms}
We will thus be concerned with the following function spaces. For $r \in \N^*$, we abusively denote by $C^r(\T^1)$ the set of $1$-periodic $C^r$ functions on $\R$, and by $\D^r(\T^1)$ the set of orientation preserving $C^r$-diffeomorphisms of $\R$ which commute with the unit translation. 
 
Now $\| \cdot \|_0:f\mapsto \sup_{\R}|f|$ defines a norm on $C^0(\T^1)$, which makes it a Banach space, as is $C^r(\T^1)$ endowed with the $C^r$-norm defined by 
$$\| \varphi \|_r = \max_{0 \le j \le r} \| D^j \varphi \|_0.$$

We will be particularly interested in the subset $C^r_0(\T^1)$ of $C^r(\T^1)$ made of the maps $\f\in C^r(\T^1)$ such that $D^l\varphi$ vanishes \emph{somewhere} for every $0 \le l \le r$. Note that for any $h\in D^{r+1}(\T^1)$, $Dh - 1$ and $\log Dh$ belong to $C^{r}_0(\T^1)$.  In particular, in the reduced setting of Theorems \ref{t:principal} and \ref{t:cantor}, for every $t\in\Z+\alpha_1\Z+\dots+\alpha_d\Z$ (resp.\ $\Z+\alpha\Z$), for every $r\in\N^*$, $Df^t-1$ and $\log Df^t$ belong to $C^r_0(\T^1)$. 

\begin{lemme}[cf.\ \cite{Yo}, Lemma 10 p.\ 349]
\label{l:10} 
Let $r\in\N$. For every $\varphi \in C^r_0(\T^1)$, 
$$\| \varphi \|_r = \| D^r \varphi \|_{0}.$$
\end{lemme}

\begin{proof} 
The inequality $\| D^r\varphi \|_{0} \le \| \varphi \|_r$ is immediate. Conversely, for $0 \le j \le r - 1$, for any $x\in \R$, there exists $y$ at distance less than $1$ from $x$ where $D^j\f$ vanishes, and the mean value theorem then implies $|D^j\f(x)|=|D^j\f(x)-D^j\f(y)|\le \| D^{j + 1} \varphi \|_0|x-y|\le \| D^{j + 1} \varphi \|_0$. Hence $\| D^j \varphi \|_0 \le \| D^{j + 1} \varphi \|_0$, which concludes the proof by a finite induction.
\end{proof}

The following proposition gives relations between the different norms $\| \cdot \|_r$ (this is precisely Proposition~3 in \cite{Yo}, for which Yoccoz refers to the appendix of \cite{Ho}). 

\begin{proposition}[Hadamard's convexity inequalities] 
\label{p:3}
Let $r_1$, $r_2$ and $r_3 \in \N$ with $r_3 \ge r_2 \ge r_1$ and $r_1 \neq r_3$. There is a constant $C$ depending only on $r_3$ such that for every $\varphi \in C^{r_3}(\T^1)$,
$$\| \varphi \|_{r_2} \le C \| \varphi \|_{r_1}^{(r_3 - r_2) / (r_3 - r_1)} \| \varphi \|_{r_3}^{(r_2 - r_1) / (r_3 - r_1)}.$$
\end{proposition}

This (applied to $r_1=0$, $r_2=p$ and $r_3=l$) directly implies the following statement, a key ingredient in the proof of the next Lemma \ref{l:12}, whose Corollary \ref{c:12}, applied to smooth times of a $C^1$ flow, gives Corollary \ref{c:log}, which will play a critical role in the calculations of the next sections. Lemma \ref{l:12} will, more generally, be very useful to control derivatives of composed maps (as one may imagine given the derivation formulas of Section \ref{sss:deriv}). Two instances of this are Lemmas \ref{l:comp} and \ref{l:short-it} (used respectively in the proofs of Proposition \ref{t:iteres} and Lemma \ref{l:prem-it} in Sections \ref{ss:iteres} and \ref{ss:gen}).

\begin{lemme}
\label{l:kl}
Let $l\in\N^*$ and let $\FF\subset C^l(\T^1)$ be a $C^0$-bounded family of maps. Then there exists $C=C(l,\FF)$ such that for every $p\in[\![0,l]\!]$ and every $\varphi\in\FF$,
$$ \|\varphi\|_{p}\le C \| \varphi \|_{l}^{p / l}.$$
\end{lemme}

\begin{lemme}[cf.\ \cite{Yo}, Lemma 12 p.\ 35]
\label{l:12} 
Let $m\le l\in\N$ and let $P$ be a polynomial in $m$ variables $X_1, \dots, X_m$, homogeneous of weight $l$ if $X_i$ has weight $i$. If $\FF\subset \D^{l+1}(\T^1)$ is a family of diffeomorphisms such that $\{\log Dg, g\in \FF\}$ is $C^0$-bounded, then there exists $C=C(l,P,\FF) > 0$ such that for every $g\in\FF$:
$$\| P(Ng, \dots, D^{m- 1} Ng) \|_0 \le C \| \log Dg \|_{l } \, ;$$
$$\| P(D^2 g, \dots, D^{m+ 1} g) \|_0 \le C \| Dg - 1 \|_{l } \, ;$$
$$\left\| P\left( \frac{D^2 g}{Dg}, \dots, \frac{D^{m+ 1} g}{Dg} \right) \right\|_0
\le C \| Dg - 1 \|_{l }.$$
\end{lemme}

\begin{remarque}
The following proof is made substantially simpler than its analogue in \cite{Yo} by the fact that we are concerned only with integral regularity. 
\end{remarque}

\begin{proof}
Let $m\le l\in\N$. It is enough to prove the estimates of the statement in the case where $P$ is a monomial $X_1^{j_1} \dots X_m^{j_m}$. The homogeneity assumption then implies $l = \sum_{p = 1}^m p j_p$. Let $j = \sum_{p = 1}^m j_p$ and $\Delta = P(Ng, \dots, D^{m- 1} Ng)$.

Lemma \ref{l:kl} gives $C\in\R$ such that, for all $p\le l$:
$$\| D^p \log Dg \|_0 \le \| \log Dg \|_{p} 
\le C \| \log Dg \|^{\frac{p}{l }}_{l }.$$
Letting $C'=C^j$, this yields
\begin{align*}
\| \Delta \|_0 & \le \prod_{p = 1}^{m} \| D^p \log Dg\|_0^{j_p} \le C' \| \log Dg \|^{\frac{(\sum_{p = 1}^{m} p j_p)}{l}}_{l} 
= C' \| \log Dg \|_{l}.
\end{align*}
The proof of the second estimate is identical, just replacing $\log Dg$ by $Dg - 1$, and the third one follows from the second one by observing that
$$P \left( \frac{D^2g}{Dg}, \dots, \frac{D^{m+ 1} g}{Dg} \right) = (Dg)^{-j} P( D^2 g, \dots, D^{m+ 1}g).$$
and using again the $C^0$-boundedness of $\{\log Dg,g\in\FF\}$.
\end{proof}

\begin{corollaire}[cf.\ \cite{Yo}, Corollary p.\ 351] 
\label{c:12} 
Let $k\in\N^*$ and let $\FF\subset \D^{k}(\T^1)$ be a family of diffeomorphisms such that $\{\log Dg, g\in \FF\}$ is $C^0$-bounded. Then there exists $C > 0$ such that for every $g\in\FF$ and every $r\in [\![0,k-1]\!]$, 
$$C^{-1} \| Dg - 1 \|_r \le \| \log Dg \|_r 
\le C\| Dg - 1 \|_r.$$ 
\end{corollaire}

\begin{proof} 
Again, the proof is a simplified version of that in \cite{Yo}. The case $r=0$ follows directly from the boundedness hypothesis on $\FF$ together with the mean value theorem applied to the maps $\exp$ and $\log$ respectively. 

Now let $r\ge 1$. Then 
$$\| \log Dg \|_r = \| D^r \log Dg \|_{0} = \| D^{r-1} Ng \|_{0} = \left\|B_r \left( \frac{D^2g}{Dg}, \dots, \frac{D^{r + 1}g}{Dg} \right)\right\|_0\le C\|Dg-1\|_r$$
where the first equality comes from Lemma \ref{l:10} (the assumption $\FF\subset\D^{k}(\T^1)$ implies that $\log Dg$ and $Dg - 1$ satisfy the hypotheses of Lemma~\ref{l:10}), the third one from Formula \eqref{e:B}, and the final inequality from the last line of Lemma \ref{l:12} above. The proof of the other inequality is similar, using formula \eqref{e:A} instead of \eqref{e:B} and the first inequality of Lemma \ref{l:12}.
\end{proof}

This applies in particular to the smooth times of the flow $(f^t)_{t\in\R}$ of a $1$-periodic $C^1$ vector field on $\R$, for which $\{\log Df^t, t\in B\}$ is indeed $C^0$-bounded for any bounded subset $B$ of $\R$ (since $t\mapsto f^t$ is continuous in $C^1$-topology), and in particular for $B=[-1,1]$, which will be sufficient for our purpose. 
 
\begin{corollaire}
\label{c:log}
Let $\xi$ be a $1$-periodic $C^1$ vector field on $\R$ and $(f^t)_{t\in\R}$ be its flow. Then for every $r\in\N$, there exists $C > 0$ such that for every $t\in[-1,1]$ for which $f^t$ is smooth, 
$$C^{-1} \| Df^{t} - 1 \|_r \le \| \log Df^{t} \|_r 
\le C\| Df^{t} - 1 \|_r.$$ 
\end{corollaire}

\begin{lemme}
\label{l:comp} 
Let $r\in\N$ and let $(\eps_s)_{s\in\N}$ be a sequence of real numbers such that $\prod_s(1+\eps_s)$ converges. Then the subset of $\D^{r+1}(\T^1)$ made of diffeomorphisms of the form $g_s\circ\dots\circ g_1$ with $s\in\N^*$ and $\|\log Dg_j\|_r\le\eps_j$ for every $j\in[\![1,s]\!]$ is $C^{r+1}$-bounded.
\end{lemme}

\begin{proof} Let $\cF$ denote this set of diffeomorphisms. Using the chain rule, one easily sees that $\{\log Dg, g\in \FF\}$ is $C^0$-bounded (by $\sum_{s=1}^{+\infty}\eps_s$). Now for every $s\in\N^*$, let $\cF_s$ be the set of diffeomorphisms of $\cF$ ``of length $s$'', \ie of the form $g_s\circ\dots\circ g_1$ with $\|\log Dg_j\|_r\le\eps_j$ for every $j\in[\![1,s]\!]$, and let $C_s := \sup_{g\in\cF_s}\|\log Dg\|_r$. We are going to show by induction on $s$ that $C_s$ is finite and that $(C_s)_{s\in\N}$ is bounded, which, together with Corollary \ref{c:12}, will yield the desired conclusion. 

For $s=1$, we directly have $C_1\le \eps_1$. Now assume $s\ge 2$, and let $g\in\cF_s$ be of the form $g_s\circ h$ with $h\in \cF_{s-1}$. Recall that, by Lemma \ref{l:10}, $\| \log Dg \|_{r}=\| D^r\log Dg \|_{0}$. Then Formula \eqref{e:G} applied to the composition $g_s\circ h$ writes $D^r \log Dg = X + Y + Z$ with:
$$X = (D^{r} \log Dg_s) \circ h \times (Dh)^{r},
\quad Y = D^{r} \log Dh \quad \text{and} $$
$$Z = \sum_{l = 1}^{r - 1}(D^{r - l} \log Dg_s) \circ h
\times (Dh)^{r - l} \times Q_l^{r}( D \log Dh, \dots, D^{l} \log Dh) .$$
First, $\| Y \|_{0} \le \| \log Dh \|_{r}\le C_{s-1}$ and $ \| X \|_{0} \le C \eps_s$ (here and from now on, $C$, $C'$, etc. denote some ``constants'' (depending on $r$ and $\eps$ but not on $g$, $s$,...) which may vary from one estimate to the next). Similarly, in $Z$, for $1 \le l \le r - 1$ :
\begin{align*}
\| (D^{r - l} \log Dg_s) \circ h \times (Dh)^{r - l} \|_{0}
& \le C \eps_s,
\end{align*}
and using Lemma \ref{l:12} :
\begin{align*}
\| G_l^r \|_0 & \le C \| \log Dh \|_l \le C \cdot C_{s-1},
\end{align*}
so
$$\| Z \|_{0}\le C \eps_s  C_{s-1}.$$
In the end, 
$$\| \log Dg \|_{r}\le \| X \|_{0}+\| Y \|_{0}+\| Z \|_{0} \le C\eps_s+ C_{s-1} (1 + C\eps_s)$$
so
$$\max(1,C_s)\le \max(1,C_{s-1}) (1 + C' \eps_s).$$
In particular, $C_s$ is finite for every $s$ by induction, and the boundedness of $(C_s)_{n\in\N}$ follows from the above inequality and the convergence of $\prod_s (1+C' \eps_s)$.
\end{proof}

\begin{lemme}
\label{l:short-it}
Let $k\in\N^*$ and let $M>0$. There exists $C>0$ such that for every $g\in\D^k(\T^1)$ and every $n\in\N$ such that $n\|\log Dg\|_0\le M$, 
\begin{equation}
\label{e:short-it}
\|\log Dg^n\|_k\le Cn\|\log Dg\|_k.
\end{equation}
\end{lemme}

\begin{proof}
We actually prove by induction (on $r$) the existence, for every $r\in[\![0,k]\!]$, of a constant $C$ such that for every $g$ and $n$ as in the statement (which will implicitly be the case from now on), 
\begin{equation}
\label{e:short-it-r}
\tag{$B_r$}
\|\log Dg^n\|_r\le Cn(\| \log Dg\|_0)^{1-\frac r k} (\|\log Dg\|_k)^{\frac r k}.
\end{equation}
Note that if one restricts to $n=1$, this follows, like Lemma \ref{l:kl}, directly from Hadamard's inequalities \ref{p:3}.

For $r=0$, \eqref{e:short-it-r} follows directly from the chain rule, with $C=1$. In particular, $g^n$ belongs to $\cF:=$ $\{h\in\D^k(\T^1);\|\log Dh\|_0\le M\}$. Now let $r \ge 1$ and assume $(B_{r'})$ is satisfied for every $r'<r$. Recall that in the present situation, $\|\log Dg^n\|_r=\| D^r \log Dg^n \|_{0}$. The derivation formula \eqref{e:H} gives:
\begin{equation}
\label{e:8}
\tag{8}
\| D^r \log Dg^n \|_{0} \le \sum_{l = 0}^{r-1} \sum_{i = 0}^{n - 1}\underbrace{\| D^{r - l} \log Dg \|_{0} \| Dg^i \|_0^{r - l} \| \tilde R_l^r \|_0}_{A_{i,l}},
\end{equation}
with
$$\tilde R_l^r = R_l^r( D \log Dg^i, \dots, D^l \log Dg^i ).$$
Again, from now on, $C$, $C'$, etc. denote some ``constants'' (depending on $r$ and $M$ but not on $g$, $n$ or~$i$) which may vary from one estimate to the next.\medskip

For $l\ge 1$, one can apply Lemma \ref{l:12} (with $\cF$ as above) to estimate $\tilde R_l^r$:
$$\| \tilde R_l^r \|_0 \le C \| \log Dg^i \|_l,$$
and thus:
$$A_{i,l}\le C \| \log Dg \|_{r - l} \| \log Dg^i \|_l$$
which, combined to the induction hypothesis, gives:
\begin{align*} 
A_{i,l}& \le C \left(C' (\| \log Dg\|_0)^{1-\frac {r-l} k}(\| \log Dg\|_k)^{\frac {r-l} k}\right)
\left(C'i (\| \log Dg\|_0)^{1-\frac {l} k}(\| \log Dg\|_k)^{\frac {l} k}\right)\\
&\le C'' \underbrace{(i \| \log Dg\|_0)}_{\le M}\left(\| \log Dg\|_0)^{1-\frac {r} k}(\| \log Dg\|_k)^{\frac {r} k}\right).
\end{align*}
For $l=0$, $(B_r)$ with $n=1$ (which is valid without need of induction, as already observed) gives a similar bound. These bounds, fed back in \eqref{e:8}, yield $(B_r)$, 
which concludes the induction.
\end{proof}

\subsection{Regularity of $\xi$ and control on $\{f^t,t\in\Z+\alpha_1\Z+\dots+\alpha_d\Z\}$}
\label{ss:conj-it}
 
Here and in the next section, we focus on Theorem \ref{t:principal}. The proof will involve some general estimates (which do not require any arithmetic condition), from which we will deduce Theorem~\ref{t:cantor} in Section \ref{ss:cantor}. \medskip

In \cite{He,Yo}, for $r\ge 1$, the $C^r$-linearisability of a circle diffeomorphism $g$ of irrational rotation number $\bar\alpha$ is reduced to the $C^r$-boundedness of the set of iterates $g^n$, $n \in \N$. Actually, if one deals with $C^\infty$ regularity and is not too concerned with the optimal regularity of the conjugacy for a given regularity assumption on $g$, one will be happy enough to know that the $C^{r+1}$-boundedness of the set of iterates \emph{implies} the $C^r$-linearisability. 

Following Herman, this is, in a nutshell, because it implies the $C^{r+1}$-boundedness of the sequence $(h_N)_N=(\frac1N\sum_{n=0}^Ng^n)_N$, which then, by Ascoli's Theorem, has a $C^r$-converging subsequence, and a simple computation shows that the limit conjugates $g$ to the corresponding rotation $R_{\bar\alpha}$. 

But one can alternatively argue as follows: if $\phi$ is a homeomorphism such that $\phi^{-1}\circ g\circ \phi=R_{\bar\alpha}$ (given by the Denjoy theorem) and if $g_t:=\phi \circ R_t\circ \phi^{-1}$ (so that $g=g_\alpha$), the $C^{r+1}$-boundedness of the set of iterates can be proved to imply, again by Ascoli, and by density of $\Z+\alpha\Z$ in $\R$, the $C^r$ regularity of every $g_t$ and the $C^r$-boundedness of the set $\{g_t, t\in [0,1]\}$ (compare with Theorem \ref{t:T} below). This can then be proved to imply the $C^r$ regularity of the conjugacy (cf.\ proof of Proposition \ref{p:ceq}) between $(g_t)_{t\in\R}$ and the group of rotations $(R_t)_{t\in\R}$. This is the argument we are going to adapt to our situation. In our setting however, the flow $(f^t)_{t\in\R}$ under scrutiny may have global fixed points, which allows a multitude of possible local behaviors, and there is no privileged model (like the action by rotations in the above argument) to conjugate it to. Nevertheless, the last step of the above argument has the following analogue in our situation. 

\begin{proposition}[Equivalent condition]
\label{p:ceq}
Let $\xi$ be a $1$-periodic $C^1$ vector field on $\R$, and let $(f^t)_{t \in \R}$ be its flow. For every $r\in\N^*$, the following are equivalent:
\begin{enumerate}
\item[(i)] $\xi$ is the pull-back of a $C^r$ vector field by an element of $\D^r(\T^1)$;
\item[(ii)] $f^t$ is $C^r$ for every $t \in \R$ and $\{f^t,t\in [0,1]\}$ is $C^r$-bounded.
\end{enumerate}
We henceforth denote by $(P_r)$ this unique property.
\end{proposition}

This is a particular case of a much more general result by Dorroh \cite{Do}, further generalized by Hart \cite{Ha,Ha2}. We give a proof based on \cite{Ha} at the end of this subsection. Of course, the implication $(i)\Rightarrow(ii)$ is not surprising: the regularity of a vector field implies the regularity of its flow. It is the converse which is of interest to us. Assuming $(ii)$, the conjugacy of $(i)$ is built as an average on $t$ of time-$t$ maps of the flow (which is actually precisely what one gets in the circle case when one takes the limit of Herman's averages $h_N$). \medskip

Now we wish to prove that, in the \emph{reduced setting} of Theorem \ref{t:principal} (cf.\ \ref{sss:reduction1}), $(P_r)$ holds for the $\xi$ under scrutiny, for any $r\in\N^*$. We will obtain $(P_r)$ in the ``second form'' $(ii)$ as follows:

\begin{theoreme}
\label{t:T}
Let $d\in\N^*$, let $(\alpha_1,\dots,\alpha_d)$ be a family of simultaneously Diophantine numbers and let $\xi$ be a $1$-periodic $C^1$ vector field on $\R$ whose time-$t$ map is $C^\infty$ for every $t\in \valpha:=\Z+\alpha_1\Z$ $+\dots+\alpha_d\Z$. Then, there exists $\cT\subset \valpha$ dense in $[0,1]$ such that, for every $r\in\N^*$, $\{f^t,t\in\cT\}$ is $C^{r+1}$-bounded.

As a consequence, for every $r\in\N^*$, $f^t$ is $C^r$ for every $t\in\R$ and $\{f^t,t\in[0,1]\}$ is $C^r$-bounded.
\end{theoreme}

In the case $d=1$, dealt with in Section \ref{sss:d1}, $\cT$ will simply be $(\Z+\alpha_1\N)\cap[0,1]$ (cf.\ Proposition \ref{t:iteres}). The situation is more subtle in the general case, studied in Section \ref{sss:d}. In these two next sections, we explain how to use the arithmetic condition to deduce Theorem \ref{t:T} from the general estimates of Lemma \ref{l:16} (which requires no arithmetic condition, and will itself be proved in Section \ref{ss:gen}).

For now, let us just explain how the first claim of the above statement implies the second: let $t \in[0,1]$ and let $(\tau_k)_k$ be a sequence in $\cT$ converging to $t$. The sequence $(f^{\tau_k})_k$ is bounded in $C^{r+1}$-topology so by Ascoli's theorem, up to extraction, $(f^{\tau_k})_k$ converges in $C^{r}$-topology towards some $h$ of class $C^r$. But $(f^{\tau_k})_k$ $C^1$-converges towards $f^t$, which proves that $f^t$ is $C^{r}$ and bounded in $C^{r}$ norm by the same constant as $\{f^t,t\in\cT\}$. The $C^r$-regularity of $f^t$ for every $t\in\R$ follows by composition.

\begin{proof}[Proof of Proposition \ref{p:ceq}]
We identify $\xi$ with the function $u\in C^1(\T^1)$ such that $\xi=u\partial_x$. In particular, then invariance of $\xi$ by its flow-maps, which normally writes $(f^t)^*\xi=\xi$, translates as 
\begin{equation}\label{e:invariance}
\xi = \frac{\xi\circ f^t}{Df^t}.
\end{equation}
Let $r\in\N^*$. \medskip

\noindent $(i)\Rightarrow(ii)$ Assume there exist $\tilde \xi\in C^r(\T^1)$ and $\psi\in \D^r(\T^1)$ such that $\xi  = \psi^*\tilde \xi$, and let $(\tilde f^t)_{t \in \R}$ be the flow of $\tilde \xi$. Then the map $(t,x)\mapsto \tilde f^t(x)$ is $C^r$, so $(ii)$ holds if one replaces $(f^t)_t$ by $(\tilde f^t)_t$. Since for every $t\in\R$, $f^t = \psi^{-1} \circ \tilde f^t \circ \psi$, one concludes using Fa\`a di Bruno's formula. \medskip

\noindent $(ii)\Rightarrow(i)$ Let us assume that $(ii)$ holds. Let 
$$\psi : x \in \R \mapsto \int_0^{1} f^s(x) ds \in \R.$$
Then $\psi$ commutes with the unit translation and $\psi$ is $C^{r}$ on $I$, by ``derivation under the integral'', thanks to the $C^r$-boundedness of the $\{f^t,t\in\R\}$, with derivatives 
$$D^k \psi(x) = \int_0^{1} D^k f^s(x) ds.$$
In particular $D\psi(x)>0$ for all $x\in\R$ since $Df^s(x)>0$ for every $s\in[0,1]$, so $\psi$ is a $C^r$-diffeomorphism. Now we need to check that $\psi$ conjugates $\xi$ to a $C^r$ vector field. Let $\tilde f^t = \psi \circ f^t \circ \psi^{-1}$ and 
$$\tilde \xi(x) = D\psi_{\psi^{-1}(x)} \xi(\psi^{-1}(x)) = \psi_* \xi (x)= \left. \frac d {dt} \right|_{t = 0} \tilde f^t(x).$$
Then actually:
\begin{align*}
\tilde \xi(x) 
& = \left. \tfrac d {dt} \right|_{t = 0} \left(  \int_0^1 f^s(f^t \circ \psi^{-1}(x)) ds \right) \\
& =  \int_0^{1} \left. \tfrac d {dt} \right|_{t= 0} (f^s \circ f^t \circ \psi^{-1})(x) ds \\
& =  \int_0^{1} Df^s(\psi^{-1}(x)) \cdot \xi(\psi^{-1}(x)) ds \\
& =  \int_0^{1} \xi(f^s(\psi^{-1}(x))) ds \quad\text{by \eqref{e:invariance}}\\
& = \int_0^{1} \tfrac d {ds} f^s(\psi^{-1}(x)) ds \\
& = f^1(\psi^{-1}(x)) - \psi^{-1}(x),
\end{align*}
which is indeed a $C^r$ map, so the proof is complete. 
\end{proof}

\begin{remarque}
Note that in the ``circle case'' where $f^1=T^1$ (cf.\ introduction), this $\tilde \xi$ is simply the unit vector field on $\R$, so $\psi$ conjugates $(f^t)_{t\in\R}$ to the group of translations $(T_t)_{t\in\R}$. Thus, if we now consider a circle diffeomorphism $g$ of irrational rotation number $\bar\alpha$, define $(g_t)_{t\in\R}$ as in the paragraph before Proposition \ref{p:ceq} and assume $\{g_t,t\in[0,1]\}$ is $C^r$-bounded, then applying the above to lifts of these diffeomorphisms, we get a $\psi\in\D^r(\T^1)$ which induces a $C^r$-diffeomorphism $\bar \psi$ of the circle conjugating $g$ to the rotation $R_{\bar\alpha}$ and defined by $\bar\psi(x)=\int_0^1g_s(x)dx$.   
\end{remarque}

\subsection{Control on $\{f^t,t\in\Z+\alpha_1\Z+\dots+\alpha_d\Z\}$ (proof of Thm. \ref{t:T})}
\label{ss:iteres}
In order to fragment the complexity, we first present the proof of Theorem \ref{t:T} for $d=1$, and we then explain how to adapt it to the general case.

\subsubsection{Case $d=1$}
\label{sss:d1}

In this case, what we prove is the exact analogue of the $C^{r+1}$-boundedness, for every $r\in\N^*$, of the set of iterates $\{g^n,n\in\N\}$ in Herman and Yoccoz' linearization Theorem for a smooth diffeomorphism $g$ with Diophantine rotation number $\bar \alpha$. Recall that, to Herman and Yoccoz' diffeomorphism $g$ corresponds, in our setting, a pair $(F^1=T_1,F^\alpha)$ of commuting diffeomorphisms of $\R$, embedded in a one-parameter family of homeomorphisms $(F^t)_{t\in\R}$ (lifts of the $g_t$'s in the previous section). The $C^{r+1}$-boundedness of $\{g^n,n\in\N\}$, is then equivalent to the $C^{r+1}$-boundedness of a particular set of lifts. Not $\{(F^\alpha)^n=F^{n\alpha},n\in\N\}$, which is not even $C^0$-bounded (the translation numbers go to infinity with $n$), but their ``translates'': $\{F^{n\alpha-[n\alpha]}=(F^\alpha)^n \circ T_{-[n\alpha]},n\in\N\}$ (where $[\cdot]$ denotes the integral part), or equivalently $\{F^{\rep{n\alpha}}, n\in\N\}$, where $\rep{\beta}$ denotes, for a real number $\beta$, not its fractional part, as is standard, but the unique representative of $\beta \mod 1$ belonging to $[-\frac 1 2, \frac 1 2)$. (This is equal to the actual fractional part when the latter is less than $\frac12$, and to the fractional part minus $1$ otherwise.)

\begin{remarque}
This notation, while unconventional, has the advantage that, if $(\frac{p_s}{q_s})_{s \in \N}$ is the sequence of \emph{convergents} of an irrational number $\alpha$, the sequence $(\{q_s \alpha\})_{s \in \N}$ goes to $0$ with alternating sign: for every $s\in\N$, $\{q_s\alpha\}=(-1)^s\|q_s\alpha\|$ and $(2q_{s+1})^{-1}<\|q_s\alpha\|<q_{s+1}^{-1}$. If $\rep{\cdot}$ was, more standardly, the fractional part, then $(\{q_s \alpha\})_{s \in \N}$ would divide into two subsequences converging respectively to $0$ and $1$. 

We refer to \cite{He}, for example, for everything about Diophantine approximation of irrational numbers. The only properties we will use are the convergence of $\prod_s(1+q_s^{-\delta})$ for every $\delta>0$ for $(q_s)$ as above, the decomposition \eqref{e:decomp} of the next page and the following consequence of the definition of a Diophantine number: if $\alpha$ is Diophantine and $(q_s)$ is the sequence of denominators of its convergents, then there exists $\nu>0$ such that $q_{s+1}\le q_s^{1+\nu}$ for every $s\in\N$.
\end{remarque}

Similarly, in our situation, what we prove is the following:

\begin{proposition}
\label{t:iteres}
Let $\xi$ be a $1$-periodic $C^1$ vector field on $\R$, whose time-$1$ and time-$\alpha$ maps of the flow $(f^t)_{t\in\R}$ are $C^\infty$, with $\alpha$ Diophantine. Then, for every $r\in\N$, $\{f^{\rep{n\alpha}},n\in\N\}$ is $C^{r+1}$-bounded.
\end{proposition}

With our notation, $\{\rep{n\alpha},n\in\N\}$ is dense in $[-\frac12,\frac12]$ rather than $[0,1]$, but the $C^{r+1}$-boundedness of $\{f^{\rep{n\alpha}},n\in\N\}$ naturally implies that of $\{f^{2\rep{n\alpha}},n\in\N\}$, which this time implies Theorem \ref{t:T} for $d=1$, letting $\cT:=\{2\rep{n\alpha},n\in\N\}\cap[0,1]$.\medskip
 
According to Corollary \ref{c:log}, it suffices to prove that, under the hypotheses of the proposition, $\{\log Df^{\rep{n\alpha}},n\in\N\}$ is $C^{r}$-bounded for every $r\in\N^*$. We will see at the end of the section that this follows, by composition, from the following, where $(\frac{p_s}{q_s})_s$ denotes the sequence of convergents of the $\alpha$ under scrutiny:

\begin{lemme} 
\label{l:15}
Under the hypotheses of Proposition \ref{t:iteres}, for every $r\in\N$, 
there exist $C$ and $\delta>0$ such that
\begin{equation*}
\forall s\in\N,\; \forall b\in [0,\tfrac{q_{s+1}}{q_s}]\cap\N,\quad 
\| \log Df^{b \rep{q_s \alpha}} \|_{r} \le Cq_s^{-\delta}.
\end{equation*}
\end{lemme}

As we will see shortly, this follows, \emph{using the Diophantine condition}, from the general estimate below, which, itself, does not require this condition, and to which the whole Section \ref{ss:gen} is devoted:

\begin{lemme} 
\label{l:16}
Let $\xi$ be a $1$-periodic $C^1$ vector field on $\R$, whose time-$1$ and time-$\alpha$ maps of the flow $(f^t)_{t\in\R}$ are $C^k$, for some irrational number $\alpha\in[0,\frac14)$ and some positive integer $k$. Then there exists $C$ such that for every $r\in [\![0,k]\!]$, 
\begin{equation*}
\forall s\in\N,\; \forall b\in [0,\tfrac{q_{s+1}}{q_s}]\cap\N,\quad \| \log Df^{b \rep{q_s \alpha}} \|_{r} \le C(q_s)^{-1+\frac r k}(q_{s+1})^{\frac r k}.
\end{equation*}
\end{lemme}

\begin{remarque}
This statement corresponds in \cite{Yo} to a combination of Lemma 14 and Proposition 5 of \cite{Yo}, which both require the Diophantine condition, while their analogues here, namely Lemma \ref{l:prem-it} and Proposition \ref{p:5}, do not (cf.\ Remarks \ref{r:C1-C0} and \ref{r:C1-C0-2}). 
\end{remarque}

Let us conclude this section by explaining how the last statement implies Lemma \ref{l:15} and how Lemma \ref{l:15} implies Proposition \ref{t:iteres}.

\begin{proof}[Proof of Lemma \ref{l:15}] 
Fix $r\in\N$. The Diophantine condition on $\alpha$ implies the existence of $\nu>0$ such that $q_{s+1}\le q_s^{1+\nu}$ for every $s\in\N$. Therefore, for a given $k\ge r$, the estimate of Lemma \ref{l:16} becomes: for every $s\in\N$ and every $b\in [0,\tfrac{q_{s+1}}{q_s}]\cap\N$,
\begin{equation*}
\| \log Df^{b \rep{q_s \alpha}} \|_{r}  \le C_k(q_s)^{-1+\frac r k}(q_s^{1+\nu})^{\frac{r}k}\le C_k q_s^{\rho(r,k)} \quad \text{with} \quad 
\rho(r,k) = \frac{r(2 + \nu)}{k} - 1.
\end{equation*}
Picking $k$ big enough that $\rho(r,k)<0$, we get the desired estimate.
\end{proof}

\begin{proof}[Proof of Proposition \ref{t:iteres}] We follow a technique already used by Denjoy, Herman, Yoccoz, etc. We still denote by $(q_s)_{s\in\N}$ the sequence of denominators of the convergents of $\alpha$ (refer to \cite{He} for the definition), and we abbreviate $\| q_s\alpha\|$ by $\alpha_s$ for every $s\in\N$ (not to be confused with the family $\alpha_1,\dots,\alpha_d$ of the general statement of Theorems \ref{t:T} or \ref{t:principal}; here, there is a single Diophantine number $\alpha$, and $(\alpha_s)_{s\in\N}$ is a sequence in $(\Z+\alpha\Z)\cap[0,\frac12)$).

Let $r\in\N^*$. Any positive integer $n$ has a canonical decomposition of the form
\begin{equation*}
\label{e:decomp}
\tag{$\diamond$}
n=\sum_{s = 0}^S b_s q_s \quad \text{ with } S \in \N^* \text{ and } \forall s\in [\![0,S]\!],\; 0 \le b_s \le a_s = [\tfrac{q_{s+1}}{q_s}]
\end{equation*}
(again, refer to \cite{He} for more details).
The tricky part compared to \cite{He,Yo} is that in general, one does not quite have 
$$\rep{n \alpha} = \sum_{s = 0}^S b_s \rep{q_s  \alpha}\quad\text{ or equivalently }\quad f^{\rep{n\alpha}}=f^{b_s \rep{q_s  \alpha}}\circ\dots\circ f^{b_0 \rep{q_0  \alpha}},$$
  as one would wish in order to deduce Theorem \ref{t:iteres} directly from Lemma \ref{l:15} using Lemma \ref{l:comp} (which controls composed maps). For example, if $b_0=a_0=q_1$ and $b_s=0$ for all $s\in\N^*$, $\sum_{s = 0}^S b_s \rep{q_s  \alpha}$ equals $a_0 \alpha = 1-\alpha_1=1+\rep{q_1\alpha}=1+\rep{n\alpha}$. 
\medskip

But in the particular case where $b_0=0$, one does have:
$$ \sum_{s = 0}^S b_s \rep{q_s  \alpha} = \sum_{s = 1}^S (-1)^s b_s \alpha_s 
\in [-\tfrac12,\tfrac12) \quad \text{so} \quad 
\rep{n \alpha} = \sum_{s = 0}^S b_s \rep{q_s  \alpha}.$$ 
Indeed, for every $s\in\N^*$, $0\le b_s\alpha_s\le a_s\alpha_s\le \alpha_{s-1}$, and $\alpha_{s+2} < \frac12\alpha_s$ (cf.\ \cite{He}) so
$$-\frac12 \le -2 \alpha \le -\sum_{t=0}^{+\infty} \alpha_{2t} 
\le \sum_{s = 0}^S b_s \rep{q_s  \alpha} 
\le \sum_{t=1}^{+\infty} \alpha_{2t-1}
\le 2 \alpha_1 < \frac12.$$

Now let us denote by $\cB$ the set of integers whose decomposition \eqref{e:decomp} starts with $b_0=0$. It is enough for us to show that $\{f^{\rep{n \alpha}}, n \in \cB\}$ is $C^{r+1}$-bounded. Indeed, any $m\in\N$ is of the form $b_0+n$ with $b_0\in[\![0,q_1]\!]$ and $n\in \cB$, in which case 
$$\rep{m\alpha} = \rep{n\alpha} +i \alpha + j \text{ with } 0 \le i \le q_1 
\text{ and } -2 \le j \le 1.$$
So 
$$\{f^{\rep{m \alpha}}, m \in\N\} \subset \bigcup_{{0\le i \le q_1} \atop {-2\le j\le 1}}
\{f^{i \alpha + j} \circ f^{\rep{n \alpha}}, n \in \cB\}$$
and each component of the finite union on the right is $C^{r+1}$-bounded if $\{f^{\rep{n \alpha}}, n \in \cB\}$ is, by composition. Now this boundedness follows directly from Lemmas \ref{l:comp} and \ref{l:15} and the fact that $\prod_s (1+q_s^{-\delta})$ converges.
\end{proof}

\subsubsection{Case $d\ge2$}
\label{sss:d}

The generalization of the proof of Theorem \ref{t:T} from the case $d=1$ to the case $d\ge2$ is entirely based on \cite{Fa-Kh}. In the case of a single Diophantine number $\alpha$ (cf.\ previous subsection), the argument leading from the general estimates of Lemma \ref{l:16} to Proposition \ref{t:iteres} (or equivalently Theorem \ref{t:T}) \emph{via} Lemma \ref{l:15} relies on the existence, for \emph{the whole sequence} of convergents of $\alpha$, of a uniform control of each denominator in terms of the previous one. In the case of a family of simultaneously Diophantine numbers $\alpha^{(1)},\dots,\alpha^{(d)}$\footnote{We use the superscript notation here to avoid confusion with the notation $\alpha_s=\|q_s\alpha\|$ of the previous section, but we will soon be reduced to the case of only two numbers anyway.}, one does not have such a control for each $\alpha^{(i)}$ individually, since each $\alpha^{(i)}$ may be Liouville. However, \emph{very roughly speaking}, the arithmetic condition guarantees the existence of so-called \emph{Diophantine strings} of denominators for which such a control holds for each $\alpha^{(i)}$ (and thus for which estimates like those of Lemma \ref{l:15} hold), strings which ``overlap'', so that in the end, we have a control on sufficiently many $f^t$, $t\in\Z+\alpha_i\Z$, $i\in[\![1,d]\!]$, to obtain Theorem \ref{t:T} by composition.\medskip

More precisely, given $\alpha\in\R\setminus\Q$ and $\nu>0$, following \cite{Fa-Kh}, we let 
$$\cA_\nu(\alpha):=\{s\in\N \,; \,q_{s+1}\le q_s^{1+\nu}\},$$
where $(q_s)_{s\in\N}$ denotes the sequence of denominators of the convergents of $\alpha$. According to \cite{Fa-Kh} (Section 5), one can reduce to proving Theorem \ref{t:principal} (and thus Theorem \ref{t:T}) in the case of a family of $d=2$ simultaneously Diophantine numbers $\alpha$ and $\beta$ \emph{in alternated configuration}. We do not need the precise definition of this notion here, only the fact that it implies the existence of $\nu>0$ such that, if $(q_s)_{s\in\N}$ and $(\tq_s)_{s\in\N}$ denote the sequences of denominators of the convergents of $\alpha$ and $\beta$ respectively, and if we define the sets of ``Diophantine times'' by
$$\cA=\{n\in\N\,;\, n=\sum_{s=0}^Sb_sq_s, \, S\in\N, \, b_s\in[0,\tfrac{q_{s+1}}{q_s}]\cap \N, \,b_s=0 \text{ if } s\notin \cA_\nu(\alpha)\}$$
$$\tilde\cA=\{n\in\N\,;\, n=\sum_{s=0}^Sb_s\tq_s, \, S\in\N, \, b_s\in[0,\tfrac{\tq_{s+1}}{\tq_s}]\cap \N, \,b_s=0 \text{ if } s\notin \cA_\nu(\beta)\}$$
then $\cT':=\{\rep{u \alpha+v\beta}, (u,v)\in\cA\times\tilde\cA\}$ is dense in $[-\frac12,\frac12]$ (cf.\ Lemma 5 p.\ 977 in \cite{Fa-Kh}). Note that if $\alpha$ is a single Diophantine number, meaning there exists $\nu>0$ such that $\cA_\nu(\alpha)=\N$, then the set $\cA$ above is $\N$, and $\{\rep{u\alpha},u\in\cA\}= (\Z+\alpha\N)\cap [-\tfrac12,\frac12)$ is indeed dense in $[-\tfrac12,\frac12]$. \medskip

We thus consider a $1$-periodic $C^1$ vector field $\xi$ on $\R$, denote its flow by $(f^t)_{t\in\R}$, and assume $f^1$, $f^\alpha$ and $f^\beta$ are smooth for $\alpha$ and $\beta$ as above, and we are left with proving that $\{\log Df^t,t\in\cT'\}$, is $C^r$-bounded, for every $r\in\N^*$, for $\cT'$ defined as above. This will follow from the next lemma:

\begin{lemme}
\label{l:15b}
For every $r\in\N^*$, there exist $\delta>0$ and $C$ such that, 
\begin{equation*}
\forall s\in \cA_\nu(\alpha),\; \forall b\in [0,\tfrac{q_s+1}{q_s}]\cap\N,\quad 
\| \log Df^{b \rep{q_s \alpha}} \|_{r} \le Cq_s^{-\delta}
\end{equation*}
\begin{equation*}
\text{and}\quad \forall s\in \cA_\nu(\beta),\; \forall b\in [0,\tfrac{\tq_s+1}{\tq_s}]\cap\N,\quad 
\| \log Df^{b \rep{\tilde q_s \beta}} \|_{r} \le C\tilde q_s^{-\delta}.
\end{equation*}
\end{lemme}
This lemma follows directly from the general estimates of Lemma \ref{l:16} by definition of $\cA_\nu(\cdot)$, just like Lemma \ref{l:15} followed from the same estimates using the Diophantine condition.\medskip

Let us finally check that this implies the $C^r$-boundedness of $\{\log Df^t,t\in\cT'\}$, just like Lemma \ref{l:15} implied Proposition \ref{t:iteres}: the same argument proves, by the very definition of $\cA$ and $\tilde\cA$, that $\{ f^{\rep{u\alpha}},u\in\cA\}$ and $\{ f^{\rep{v\beta}},v\in\tilde\cA\}$ are $C^{r+1}$-bounded. Now if $(u,v)\in\cA\times\tilde\cA$, $\rep{u\alpha+v\beta} = \rep{u\alpha}+\rep{v\beta}+0$, $1$ or $-1$, and one concludes by composition. 

\subsection{General estimates (proof of Lemma \ref{l:16})}
\label{ss:gen}

We place ourselves under the hypothesis of Lemma \ref{l:16}, namely: $\alpha\in(\R\setminus\Q)$ $\cap [0,\frac14)$, $\xi$ is a $1$-periodic $C^1$ vector field on $\R$, $(f^t)_t$ is its flow and $f^1$ and $f^\alpha$ are assumed $C^k$ for some $k\in\N^*$. Recall that, in this situation, for every $t\in\Z+\alpha\Z$, for every $r\in\N^*$, $Df^t-1$ and $\log Df^t$ belong to $C^r_0(\T^1)$ (cf Section \ref{ss:norms}). We still denote by $(q_s)_{s\in\N}$ the sequence of denominators of the convergents of $\alpha$, and write, for every $s\in\N$, $\alpha_s=\|q_s\alpha\|=|\rep{q_s\alpha}|$, which satisfies $(2q_{s+1})^{-1}\le\alpha_s\le (q_{s+1})^{-1}$. The central claim of this section is that, under these hypotheses:

\begin{proposition}[cf.\ \cite{Yo}, Proposition 5 p.\ 355] 
\label{p:5}
The sequence $(\| \log Df^{ \rep{q_s \alpha}} \|_{k} / q_s)_{s \in \N}$ is bounded.
\end{proposition}

In what follows, we let
 $$\Delta_s^{(k)}:=\| \log Df^{ \rep{q_s \alpha}} \|_{k}.$$
(beware that the $\Delta_s$ of \cite{Yo} corresponds to our $\Delta_s^{(k-1)}$).\medskip

The above proposition is proved at the end of this section by a kind of induction on $s\in\N$, decomposing $f^{\rep{q_{s+1} \alpha}}$ in terms of $f^{\rep{q_{s} \alpha}}$ and $f^{\rep{q_{s-1} \alpha}}$ and \emph{using} the estimates of Lemmas \ref{l:13} and \ref{l:prem-it} below (which are themselves easily obtained from Lemma \ref{l:C1} about the first derivatives using Hadamard's convexity inequalities \ref{p:3}). The resulting bound is then \emph{fed back} in these very estimates, and the outcome is precisely Lemma \ref{l:16}, which concludes the proof of Theorem \ref{t:principal}. 
\begin{lemme}
\label{l:C1}
There exists $C>0$ such that 
\begin{equation*}
\forall t\in\R, \quad \| \log Df^{t} \|_0 \le C |t|
\end{equation*}
and in particular
\begin{equation*}
\label{e:C1}
\forall s\in\N, \quad \| \log Df^{ \rep{q_s \alpha}} \|_0 \le C |\rep{q_s\alpha}| = C \alpha_s.
\end{equation*}
\end{lemme}

\begin{proof} 
This is where starting with a $C^1$ vector field is important: for every $x\in \R$, $t\in\R\mapsto Df^t(x)$ is a solution of the ``first variations'' equation: 
$$\frac d{dt}  Df^t(x) = D\xi(f^t(x))\, Df^t(x)\quad\text{with}\quad Df^0(x)=1,$$
or equivalently, since $t\mapsto Df^t(x)$ does not vanish,  
$$\frac d{dt} \log Df^t(x) = D\xi(f^t(x))\quad\text{with}\quad \log Df^0(x)=0,$$
so for every $t\in\R$, $\|\log  Df^t\|_0\le \|D\xi\|_0 |t|$.
\end{proof}

\begin{remarque}\label{r:C1-C0}
In \cite{Yo}, the ``analogue'' of the above statement (Proposition 4 p.\ 352 with $\gamma_0=0$) gives, with our notations, a control on $\|f^{ \rep{n \alpha}}-\id\|_0$ rather than $\|\log Df^{ \rep{n \alpha}}\|_0$. This is because the hypothesis of a $C^1$-conjugacy with which one starts the bootstrap there is weaker than our hypothesis of a $C^1$ generating vector field. This shift of one degree of regularity is reflected in the exponents appearing in the bounds for higher derivatives obtained using Hadamard's convexity inequalities. Namely, in the analogues of the following Lemmas \ref{l:13} and \ref{l:prem-it} in \cite{Yo} (Lemmas 13 and 14), $\frac r k$ is replaced by $\frac{r+1}{k+1}$. This difference turns out to have a tremendous effect on the rest of the proof. Namely, in \cite {Yo}, in the proof of Proposition~5 -- the analogue of our Proposition \ref{p:5} -- the Diophantine condition is used to compensate the ``inadequacy'' of the exponents, whereas we do not need it here, with the consequence that the estimates we get are independent of the arithmetic condition and in particular allow us to prove Theorem \ref{t:cantor}. 
\end{remarque}

\begin{lemme}[cf.\ \cite{Yo}, Lemma 13 p.\ 353] 
\label{l:13}
There exists $C > 0$ such that for every $r \in [\![0, k]\!]$ and every $s \in \N$, 
$$\| \log Df^{\rep{q_s \alpha}} \|_{r} 
\le C q_{s + 1}^{-1}(q_{s + 1} \Delta_s^{(k)})^{\frac {r} k}.$$
\end{lemme}

\begin{proof} 
Hadamard's inequalities (applied to $r_1 = 0$, $r_2 = r$ and $r_3 = k$) and Lemma \ref{l:C1} give constants $C$ and $C'$ such that, for every $r \in [\![0, k]\!]$ and every $s \in \N$,
\begin{align*}
\| \log Df^{\rep{q_s \alpha}} \|_{r}
& \le C \left( \| \log Df^{\rep{q_s \alpha}} \|_0 \right)^{1 - {\frac {r} k}} 
\left( \| \log Df^{\rep{q_s \alpha}}  \|_{k} \right)^{\frac {r} k} \\
& \le C' \alpha_s^{1 -{\frac {r}  k}} (\Delta_s^{(k)})^{\frac {r}  k} 
\le C' (q_{s + 1})^{\frac {r}  k-1}(\Delta_s^{(k)})^{\frac {r}  k}.
\end{align*}
\end{proof}

The above estimates generalize to ``small'' iterates using the general Lemma \ref{l:short-it}:

\begin{lemme}[cf.\ \cite{Yo}, Lemma 14 p.\ 353]
\label{l:prem-it} 
There exists $C>0$ such that for every $r  \in [\![0, k]\!]$, every $s\in\N$ and every $n\in[0,\frac{q_{s + 1}}{q_s}]\cap\N$,
\begin{equation*}
\label{e:Er}
\tag{$E_r$}
\| \log Df^{n\rep{q_s \alpha}} \|_{r} \le C q_{s}^{-1}(q_{s + 1} \Delta^{(k)}_s)^{\frac {r} k}.
\end{equation*}
\end{lemme}

\begin{proof} Observe that the desired upper bound is obtained by multiplying that of the preceding Lemma by $\frac{q_{s+1}}{q_s}$, to which $n$ is assumed inferior.  So one simply needs to apply the general Lemma \ref{l:short-it}, observing that, for $g=f^{\rep{q_s \alpha}}$ and $n\in[0,\frac{q_{s + 1}}{q_s}]\cap\N$, 
\begin{equation}
\label{e:n-log}
n \|\log Dg\|_0\le \frac{q_{s+1}}{q_s} \,C q_{s+1}^{-1}<C,
\end{equation}
with $C$ given by Lemma \ref{l:C1}.
\end{proof}

\begin{remarque}
\label{r:C1-C0-2}
In \cite{Yo}, the deduction of Lemma 14 from Lemma 13 (the analogues of Lemmas \ref{l:prem-it} and \ref{l:13} respectively) is less straightforward, partly because one does not have the analogue of \eqref{e:n-log} above, which allows to apply the general Lemma \ref{l:short-it} (though one does have a uniform bound on $\|\log Dg^n\|_0$, but this is weaker). So one needs to apply the general scheme of Lemma \ref{l:short-it} directly to the diffeomorphisms of Lemma 13/\ref{l:13} and to throw in the Diophantine condition in order to compensate the ``insufficiency'' of the bounds given by Hadamard's convexity inequalities, as in Remark \ref{r:C1-C0}. 
\end{remarque}

\begin{proof}[Proof of Proposition \ref{p:5}]
Throughout this proof, we abbreviate $\Delta_s^{(k)}$ by $\Delta_s$. We recall that $(\sum \alpha_s)$ is a converging series. For every $s\in \N^*$, let $R_s = \sum_{t = s}^{+ \infty} \alpha_t$ be its remainder at order $s$, and define
$$\Delta'_s = \sup \{ \| (D^{k} \log D f^{\rep{q_t \alpha}}) \circ f^{u} \times (Df^{u})^{k } \|_0 ; (t,u)\in\N\times \R, 0 \le t \le s, |u| \le R_{s-1}\}.$$
Then $\Delta_s \le \Delta'_s $ ($\Delta_s= \| D^{k} \log Df^{\rep{q_s \alpha}} \|_0$ belongs to the set of which $\Delta_s'$ is the supremum), so it is sufficient to prove that $(\Delta'_s / q_s)_{s\in\N}$ is bounded. 

Fix $s \in \N^*$. We have $\rep{q_{s + 1} \alpha} =  \rep{q_{s - 1} \alpha} + a_s \rep{q_s \alpha}$, with $a_s \ge 1$, so according to the derivation formula \eqref{e:G}, we have a decomposition:
$$D^{k} \log Df^{\rep{q_{s + 1} \alpha}} 
= D^{k} \log Df^{\rep{q_{s - 1} \alpha} + a_s \rep{q_{s} \alpha}} 
= X + Y + Z$$
with 
$$X = (D^{k} \log Df^{\rep{q_{s - 1} \alpha}}) \circ f^{a_s \rep{q_{s} \alpha}} 
\times (Df^{a_s \rep{q_{s}\alpha}})^{k}, 
\qquad Y = D^{k} \log Df^{a_s \rep{q_{s} \alpha}} \quad \text{and}$$
\begin{align*}
Z = \sum_{l = 1}^{k - 1}(D^{k - l} \log Df^{\rep{q_{s - 1} \alpha}})
& \circ f^{a_s \rep{q_{s} \alpha}} \times (Df^{a_s \rep{q_{s} \alpha}})^{k - l} \\
& \times Q_l^{k}(D\log Df^{a_s \rep{q_{s} \alpha}}, \dots, D^{l} \log Df^{a_s \rep{q_{s} \alpha}}).
\end{align*}
Let us write, furthermore, for any given $u \in [-R_{s}, R_{s}]$,
\begin{align*}
X' & = (X \circ f^{u}) (Df^{u})^{k},\\
Y' & = (Y \circ f^{u}) (Df^{u})^{k},\\
Z' & = (Z \circ f^{u}) (Df^{u})^{k},
\end{align*}
so that
$$(D^{k} \log Df^{\rep{q_{s + 1} \alpha}}) \circ f^{u} (Df^{u})^{k} = X' + Y' + Z'.$$
We must thus estimate $\| X' \|_0, \| Y' \|_0$ and $\| Z' \|_0$ as finely as possible (for example, writing $\|X'\|_0\le \|X\|_0 \|Df^u\|_0$ is already a loss of precision we cannot afford).\medskip

First,
$$\| X' \|_0 = \left\|(D^{k} \log Df^{\rep{q_{s - 1} \alpha}}) \circ f^{a_s \rep{q_{s} \alpha} + u} \times (Df^{a_s \rep{q_{s} \alpha} + u})^{k} \right\|_0,$$
and since $| a_s \rep{q_{s} \alpha} | = a_s \alpha_s \le \alpha_{s - 1}$ and $|u| \le  R_s$, $| u + a_s \alpha_s | \le R_s + \alpha_{s - 1} = R_{s - 1}$ so 
$$\| X' \|_0 \le \Delta_{s - 1}'.$$

The way in which we are now going to control $\| Y' \|_0$ is very similar to the proof of Lemma \ref{l:short-it}. We apply Formula \eqref{e:H} to $g = f^{\rep{q_{s} \alpha}}$:
\begin{align*}
Y' = \sum_{l = 0}^{k-1} \sum_{i = 0}^{a_s - 1} ( D^{k - l} \log Df^{\rep{q_{s} \alpha}} ) \circ f^{i \rep{q_{s}\alpha} + u} \times ( Df^{i \rep{q_{s} \alpha} + u} )^{k - l} \times \tilde R_l^{k } = \sum_{l = 0}^{k-1} Y'_l
\end{align*}
with 
$$\tilde R_0^{k} = 1, \quad \tilde R_l^{k} = R_l^{k}(D \log Df^{i \rep{q_{s} \alpha}}, \dots, D^l \log Df^{i \rep{q_{s} \alpha}}) \circ f^u (Df^u)^l \quad \text{for $l > 0$}.$$
Again, for $i\in[\![0,a_s-1]\!]$, $|i\rep{q_{s}\alpha} + u|\le  R_{s-1}$, so one directly has $\| Y_0' \|_0 \le a_s \Delta_s'$. As for $l\ge1$, 
\begin{align*}
\| \tilde R_l^{k} \|_0 
& \le C \| \log Df^{i \rep{q_{s} \alpha}} \|_l\quad \text{ by Lemmas \ref{l:12} and \ref{l:C1},}  \\
& \le C' q_s^{-1} (\Delta_s q_{s + 1})^{\frac{l}k}\quad \text{ by Lemma \ref{l:prem-it}}  
\end{align*}
(here and from now on, as in previous proofs, $C$, $C'$, etc. denote some ``constants'' (depending on $k$, $l$ and $\xi$ but not on $s$, $i$ or $u$) whose value may vary from one estimate to the next). Hence, using Lemmas \ref{l:C1} and \ref{l:13},
$$\| Y_l' \|_0 \le C a_s \left( q_{s + 1}^{-1} (\Delta_s q_{s + 1})^{\frac{k-l}k} \right) \left( q_s^{-1}(\Delta_s q_{s + 1})^{\frac l k}\right) = C a_s \Delta_s q_s^{-1}\le C a_s\Delta_s' q_s^{-1},$$
and in the end:
$$\| Y' \|_0 \le a_s \Delta_s' (1 + C q_s^{-1}).$$
In order to estimate $\| Z' \|_0$, we have, for $1 \le l \le k$, according to Lemmas \ref{l:12}, \ref{l:C1} and \ref{l:prem-it} 
\begin{align*}
\left\| \left( Q_l^{k}(D \log Df^{a_s \rep{q_{s} \alpha}}, \dots, D^{l} \log Df^{a_s \rep{q_{s} \alpha}}) \circ f^{u}  \right) (Df^{u})^l \right\|_0
& \le C \| \log Df^{a_s \rep{q_{s} \alpha}} \|_l\\
& \le C q_s^{-1} (\Delta_s q_{s + 1})^{\frac{l}k}
\end{align*}
and according to Lemmas \ref{l:C1} and \ref{l:13},
\begin{align*}
\left\| (D^{k  - l} \log Df^{\rep{q_{s - 1} \alpha}}) \circ f^{a_s \rep{q_{s} \alpha}} \times (Df^{a_s \rep{q_{s}\alpha}})^{k  - l} \right\|_0
& \le C q_{s}^{-1}(\Delta_{s - 1} q_{s})^{\frac{k-l}k}\\
& \le C q_{s}^{-1}(\Delta_{s} q_{s+1})^{\frac{k-l}k}.
%$\Delta_{s - 1}\le \Delta_s$, $q_s\le q_{s + 1}$
\end{align*}
Thus,
$$\| Z' \|_0 \le C' q_s^{-2}q_{s + 1}\Delta_s \le C'' q_s^{-1}a_s \Delta_s'.$$
Gathering the estimates on $X'$, $Y'$ and $Z'$, we get:
\begin{align*}
\| (D^{k} \log Df^{\rep{q_{s + 1} \alpha}}) \circ f^{u} (Df^{u})^{k} \|_0 
& \le \Delta_{s - 1}' + a_s \Delta_s' (1 + C q_s^{-1}) \\
& \le \max \left(\frac{\Delta_{s - 1}'}{q_{s - 1}}, \frac{\Delta_{s}'}{q_{s}} \right) ((q_{s - 1} + a_s q_s) + a_s q_s C q_s^{-1})\\
& \le \max \left( \frac{\Delta_{s - 1}'}{q_{s - 1}}, \frac{\Delta_{s}'}{q_{s}} \right) q_{s + 1} (1 + C q_s^{-1}).
\end{align*}
Let $\theta_s = \max \{\Delta'_t / q_t, 0 \le t \le s\}$. We just proved the existence of $C\in\R$ such that for all $s\ge 1$,
$$\theta_{s + 1} \le \theta_s (1 + C q_s^{-1}).$$
Now $\prod_s (1 + C q_s^{-1})$ is converging, so the sequence $(\theta_s)_s$ is bounded, which concludes the proof.
\end{proof}

\subsection{Proof of Theorem \ref{t:cantor}}
\label{ss:cantor}

The estimate of Lemma \ref{l:13} combined to the result of Proposition \ref{p:5} yields:

\begin{lemme}
\label{l:cantor}
Let $\xi$ be a $1$-periodic $C^1$ vector field on $\R$, whose time-$1$ and time-$\alpha$ maps of the flow $(f^t)_{t\in\R}$ are $C^\infty$, for some irrational number $\alpha\in[0,\frac14)$. Then, for every $r\in\N^*$, there exist $\delta>0$ and $C\in\R$ such that 
$$\forall s\in\N,\quad \| \log Df^{\rep{q_s \alpha}} \|_{r} \le C q_{s + 1}^{-\delta}.$$
In particular (applying Corollary \ref{c:12}), the sequence $(f^{\rep{q_s \alpha}})_{s\in\N}$ converges towards the identity in $C^r$-topology.
\end{lemme}

\begin{proof} 
Fix $r\in\N^*$ and let $k\in \N$ be such that $1-\frac {2r} k=\delta>0$. Lemma \ref{l:13} and Proposition \ref{p:5} then give $C$ and $M>0$ such that for every $s \in \N$, 
$$\| \log Df^{\rep{q_s \alpha}} \|_{r} \le C q_{s + 1}^{-1}(q_{s + 1} M\underbrace{q_s}_{\le q_{s+1}})^{\frac {r} k}\le CM^{\frac r k}(q_{s+1})^{-1+\frac{2 r}k}\le C'(q_{s+1})^{-\delta}\xrightarrow[s\to+\infty]{}0.$$
\end{proof}

From now on, we place ourselves in the reduced setting of Theorem \ref{t:cantor}, \ie precisely under the hypotheses of the above lemma. Letting 
$$K_\alpha=\left\{\sum_{s=1}^{+\infty}b_s \rep{q_s \alpha}; (b_s)_{s\in\N^*}\in\{0,1\}^{\N^*}\right\},$$
let us prove that $f^t$ is $C^\infty$ for every $t$ in the Cantor set $K_\alpha$. Deducing this from Lemma \ref{l:cantor} is very similar to deducing Proposition \ref{t:iteres} and its corollary Theorem \ref{t:T} (case $d=1$) from Lemma \ref{l:15}. 

Fix $(b_s)_{s\in\N^*}\in\{0,1\}^{\N^*}$, let $\tau=\sum_{s=1}^{+\infty}b_s \rep{q_s \alpha}$ and, for every $S\in\N^*$, let $\tau_S=\sum_{s=1}^{S}b_s \rep{q_s \alpha}$ and $n_S=\sum_{s=1}^S b_sq_s$, so that, as seen in the proof of Proposition \ref{t:iteres}, $\tau_S=\rep{n_S\alpha}$. Hence $(f^{\rep{n_S\alpha}})_{S\in\N^*}$ converges in $C^1$-topology towards $f^\tau$, so by Ascoli, it suffices to prove that $\{f^{\rep{n_S\alpha}}, S\in\N^*\}$ is $C^{r+1}$-bounded, or equivalently (by Corollary \ref{c:log}), that  $\{\log Df^{\rep{n_S\alpha}}, S\in\N^*\}$ is $C^{r}$-bounded, for every $r\in\N^*$. The proof is then identical to that of Proposition \ref{t:iteres}, except this time $(b_s)_{s\in\N^*}$ belongs to $\{0,1\}^{\N^*}$, Lemma \ref{l:cantor} is used instead of \ref{l:15}, and $q_s^{-\delta}$ is replaced by $q_{s+1}^{-\delta}$, which is even better.

\section{Proof of Theorem \ref{t:liouville}}
\label{s:liouville}
 
\subsection{Overview}
 
Most statements of this section will be
made precise and proved afterwards, in Sections \ref{ss:manu} to~\ref{ss:conv}. In the introduction, we made a parallel between Theorem \ref{t:liouville} and Theorem \ref{t:H} about circle diffeomorphisms. In the following outline, the ideas which are specific to the closed half-line, as opposed to the circle (or in fact $\R$), are developed in Section \ref{sss:sergeraert}, while Section \ref{sss:AK} deals with ideas common to both. 

\subsubsection{Sergeraert's construction}
\label{sss:sergeraert}

It is already not obvious to build a $\CC^1$ contracting vector field on $\R_+$ whose flow contains \emph{some} smooth time-$t$ maps 
and some non-smooth ones. This is what Sergeraert does in \cite{Se}, with a smooth time-$1$ map and a non-$C^2$ time-$\frac12$ map.\medskip

Note that, if we work on $\R_+^*$ rather than $\R_+$, it is easy to construct a non-vanishing $C^1$ vector field whose time-$1$ map is $C^\infty$ while its time-$\frac12$ map is not. Indeed, first identify $\R_+^*$ with $\R$. Now start with the unit vector field on $\R$, whose time-$1$ map is the unit translation $T_1$. Pull it back by any $C^2$ and non-$C^3$ diffeomorphism $\Phi$ commuting with the unit translation but not at all with the translation by $\frac12$. For example, take $\Phi=\id$ on $[0,\frac12]$ and $\Phi\neq \id$ and not $C^3$ on $[\frac12,1]$. The new vector field is $C^1$ as a $C^2$-pull back of a smooth vector field. The new flow maps are conjugated to the old ones by $\Phi$. The time-$1$ map is thus left unchanged, while the new time-$\frac12$ map is not $C^3$ anymore.  \medskip

But this idea fails (at least without adaptation) on the \emph{closed} half-line: if we start with a smooth contracting vector field on $\R_+$, and we pull it back by a $C^2$-diffeomorphism $\Phi$ which commutes with the time-$1$ map, then by Kopell's Lemma, $\Phi$ belongs to the flow, so the result of the pull-back is... the initial vector field! Nevertheless, we will see shortly that the idea in Sergeraert's construction (though this is not explicit in his formulation) and in ours is indeed to start with a smooth vector field and to perform successive pull-backs by diffeomorphisms which ``almost'' commute with the time-$1$ map but do not commute at all with other times of the flow. It was pointed out to us by C. Bonatti that this construction can be slightly modified so that the conjugating diffeomorphisms in the sequence have disjoint supports, which makes the computations much simpler. What we describe now is this ``variation on Sergeraert's construction''. \medskip

We start with a smooth vector field $\xi_0$ (described below). We are going to obtain the desired vector field $\xi$ (the one with a smooth time-$1$ map and a non-$\CC^2$ time-$1/2$ map) as a limit of a sequence of deformations $\xi_k$, $k\in\N^*$. Each $\xi_k$ is defined as a pull-back $\f_k^*\xi_{k-1}$ for some smooth diffeomorphism $\f_k$ of $\R_+$  supported in an interval $I_k$ closer and closer to $0$ as $k$ grows and containing many fundamental intervals of the time-$1$ map $f_0^1=:f_0$ of $\xi_0$. Moreover, we are going to choose these supports pairwise disjoint and sufficiently far away from one another so that the relation $f_k^t = \f_k^{-1} \circ f_{k-1}^t \circ \f_k$ between the flows $(f_k^t)_{t\in\R}$ and $(f_{k-1}^t)_{t\in\R}$ of $\xi_k$ and $\xi_{k-1}$ becomes, for $t\in[0,1]$,
$$f_k^t -f_{k-1}^t=\f_k^{-1} \circ f_0^t \circ \f_k - f_0^t \text{ on $I_k\cup f_{0}^{-t}(I_k)$}\quad \text{and}\quad f_k^t = f_{k-1}^t \text{ elsewhere. } $$

The point is to cook up the conjugations $\f_k$ so that $(f_k^1)_k$ converges in $\Cinf$-topology while $(f_k^{1/2})_k$ converges only in $\CC^1$-topology (in particular, $(\chi_k)_k=(\f_1\circ\dots\circ \f_k)_k$ must diverge in $\CC^2$-topology). Thus, what we really want is $\f_k^{-1} \circ f_0^1 \circ \f_k - f_0^1$ to be $\CC^k$-small (say less than $2^{-k}$) while $\f_k^{-1} \circ f_0^{1/2} \circ \f_k - f_0^{1/2}$ is $\CC^2$-big. 

We now explain how this can be achieved with a $\f_k$ commuting with $f_{0}^1$ almost everywhere (outside two fundamental intervals of this map, to be precise) but not at all with $f_{0}^{1/2}$, provided the initial vector field is cleverly chosen (recall we cannot ask $\varphi_k$ to commute with $f_0^1$ everywhere, otherwise it would also commute with $f_{0}^{1/2}$). Namely, the $\xi_0$ we start with is made of ``bricks'' of the form described on Figure \ref{fig:Bk} (which actually represents the graph of the function $dx(\xi_0)$) defined on smaller and smaller pairwise disjoint segments $B_k$, $k\in\N^*$, closer and closer to $0$ as $k$ grows (which will contain the  $I_k$, $k\in\N^*$, mentioned above) and glued together smoothly by interpolation on the complementary intervals $G_k$ (cf.\ Figure \ref{fig:xi0}, and \eqref{e:nu0} in Section \ref{ss:manu} for the actual definition). 
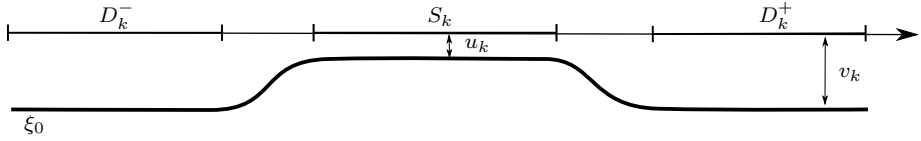
\begin{figure}[h!]
\centering
\includegraphics[width=12cm]{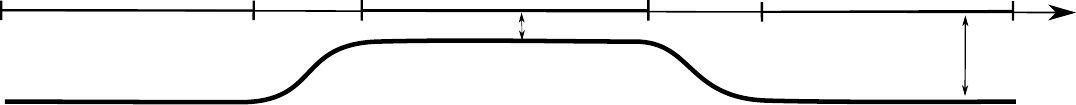}
\put(-170,24){$\scriptstyle{u_k}$}
\put(-60,34){$\scriptstyle{D_k^+}$}
\put(-307,34){$\scriptstyle{D_k^-}$}
\put(-184,34){$\scriptstyle{S_k}$}
\put(-30,13){$\scriptstyle{v_k}$}
\put(-335,-7){$\scriptstyle{\xi_0}$}
\caption{A ``brick'' of $\xi_0$}
\label{fig:Bk}
\end{figure}

Each brick resembles an undersea landscape with a shallow central region and symmetric equally deep regions. Their respective altitudes $-u_k$ and $-v_k$ (measured from the water surface, so that $0 < u_k < v_k$) go to zero much faster than their widths (so that $\xi_0$ is infinitly flat at $0$), but at very different speeds in the sense that the ratios $ v_k / u_k $ (and actually $v_k^k / u_k$) tend to infinity. As we will see, this vector field is specifically designed so that a small and very localized perturbation of $f_0^1$ in the ``deep regions'' (resulting from a conjugation) translates into a huge perturbation of its Szekeres vector field and some of its flow-maps in the ``shallow ones''. 
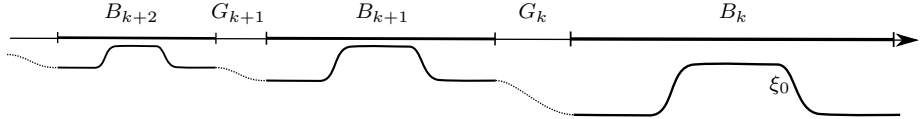
\begin{figure}[htbp]
\centering
\includegraphics[width=12cm]{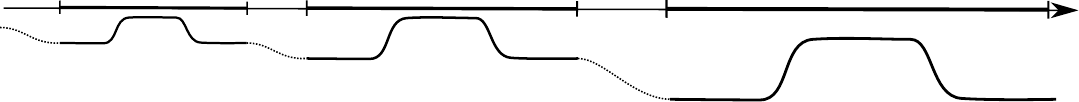}
\put(-56,10){$\scriptstyle{\xi_0}$}
\put(-75,37){$\scriptstyle{B_{k}}$}
\put(-150,37){$\scriptstyle{G_k}$}
\put(-211,37){$\scriptstyle{B_{k+1}}$}
\put(-265,37){$\scriptstyle{G_{k+1}}$}
\put(-305,37){$\scriptstyle{B_{k+2}}$}
\caption{Bricks and gluing regions.}
\label{fig:xi0}
\end{figure}

Let us thus move on to the description of $\f_k$, for $k\ge 1$. It will be the identity except on a subinterval $I_k$ of the domain $B_k$ of the $k$-th brick, where it will coincide with a diffeomorphism $\phi_k$ of $\R_+^*$ also commuting with $f_0^1$ but oscillating wildly, especially in the domain $S_k$ of the ``shallow region''. More precisely, the ingredients are the following:
\begin{itemize}
\item let $\psi$ be the $C^\infty$-diffeomorphism from $\R$ to $(0,+\infty)$ defined by $\psi(s)=f_0^s(1)$. It conjugates the restriction to $(0,+\infty)$ of each flow map $f_0^t$ to the translation $T_t$ by $t$ on $\R$, and in particular $f_0^1$ to the unit translation, and it sends $\Z$ to the orbit of $1$ under $f_0^1$. Conjugation by $\psi$ thus yields a one-to-one correspondence between diffeomorphisms of $\R$ commuting with the unit translation and diffeomorphisms of $\R_+^*$ commuting with $f_0^1$. Note by the way that 
$D\psi=\xi_0\circ \psi$, or equivalently $D\psi^{-1}=1/\xi_0$. In particular, in restriction to the domains $D_k^\pm$ of the ``deep regions'' (resp.\ to $S_k$), $\psi^{-1}$ is a homothety of ratio $-v_k^{-1}$ (resp.\ $-u_k^{-1}$);
\item let $\delta_k$ be the smooth $1$-periodic map on $\R$ whose restriction to $[0,1]$ is described in Figure~\ref{fig:gammak}, and let $\Phi_k=\id_\R+\delta_k$. In particular, $\Phi_k$ fixes $\frac14\Z$ since $\delta_k$ vanishes at $0$, $\frac14$, $\frac12$ and $\frac34$, and $\Phi_k$ is infinitely tangent to the identity at every point of $\frac12\Z$ ($\Phi_k$ is the analogue of $\Phi$ in the ``baby case'' of the second paragraph of this section);
\item let $\phi_k=\psi\circ \Phi_k\circ \psi^{-1}$. It commutes with $f_0^1$ on $\R_+^*$ since $\Phi_k$ commutes with $T_1$ on $\R$ (but not with $f_0^{1/2}$ as we will see), fixes the orbit of $1$ under $f_0^{1/4}$ and is infinitely tangent to the identity (ITI) at every point of the orbit of $1$ under $f_0^{1/2}$. We will see that the shape of $\xi_0$ and size of $\delta_k$ are precisely designed so that $\phi_k-\id$ is $C^k$-small on $D_k^\pm$ and $C^2$-big on $S_k$ (cf.\ \eqref{e:estimphim});
\item let $x_k^-$ and $x_k^+$ be two elements of the orbit of $1$ under $f_0^1$ lying ``in the middle of'' $D_k^-$ and $D_k^+$ respectively, so that $J_k^\pm=[x_k^\pm,f_0^{-1}(x_k^\pm)]$ lies entirely in $D_k^\pm$, and $x_k^-=f_0^{n_k}(x_k^+)$ for some (big) $n_k\in\N$ (cf.\ \eqref{e:xn}).
\end{itemize} 
\begin{figure}[h!]
\centering
\includegraphics[width=6cm]{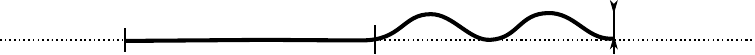}
\put(-90,-5){$\scriptstyle{1/2}$}
\put(-145,-5){$\scriptstyle{0}$}
\put(-30,10){$\scriptstyle{\frac{u_k}{v_k}}$}
\put(-34,-7){$\scriptstyle{1}$}
\caption{$\delta_k$ on $[0,1]$}
\label{fig:gammak}
\end{figure}
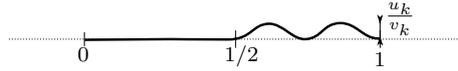
We define $\f_k$ as $\phi_k$ on $I_k:=[x_k^-,x_k^+]$ and the identity elsewhere (it is smooth since $\phi_k$ is ITI at $x_k^-$ and $x_k^+$). In particular it is supported in $B_k$ as required. 
\begin{figure}[h!]
\centering
\includegraphics[width=14cm]{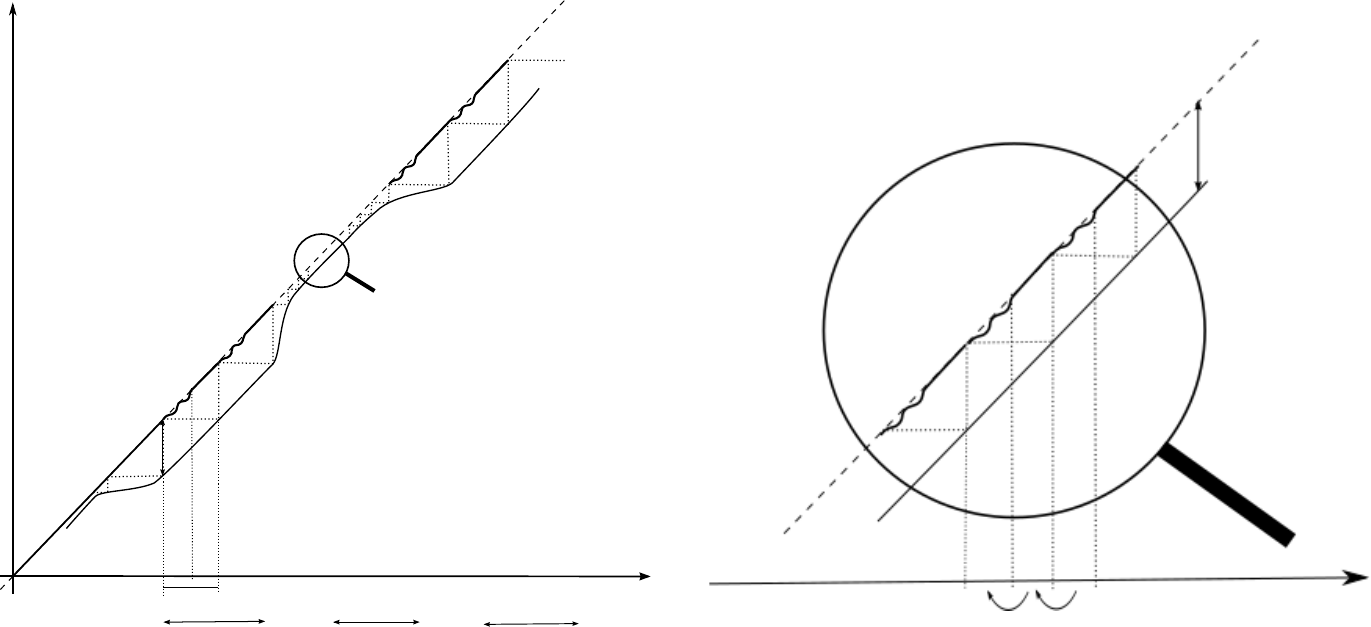}
\put(-362,7){$\scriptstyle{x_k^-}$}
\put(-346,5){$\scriptstyle{J_k^-}$}
\put(-340,-8){$\scriptstyle{D_k^-}$}
\put(-248,-8){$\scriptstyle{D_k^+}$}
\put(-292,-8){$\scriptstyle{S_k}$}
\put(-108,108){$\scriptstyle{\f_k}$}
\put(-71,95){$\scriptstyle{f_0^1}$}
\put(-98,-4){$\scriptstyle{f_0^{1/2}}$}
\put(-340,88){$\scriptstyle{\f_k}$}
\put(-330,59){$\scriptstyle{f_0^1}$}
\put(-359,48){$\scriptstyle{v_k}$}
\put(-45,142){$\scriptstyle{u_k}$}
\caption{Shape of $\f_k$.}
\label{fig:gk}
\end{figure}

As a first observation, $\phi_k$ and $\id_{\R_+}$ commute with $f_0^1$, so $\f_k$, which is ``piecewise one or the other'', commutes with $f_0^1$ except ``near the transitions''. More precisely, one can check (cf \eqref{e:f0pqk} and its proof) that by construction $\gamma_k:=f_k^1-f_{k-1}^1=\f_k^{-1}\circ f_{k-1}^1\circ \f_k-f_{k-1}^1$ vanishes outside $[x_k^-,f_0^{-1}(x_k^+)]$, that it is equal there to $\f_k^{-1} \circ f_0^1 \circ \f_k -f_0^1$, and that this vanishes outside the fundamental intervals $J_k^\pm=[x_k^\pm,f_0^{-1}(x_k^\pm)]$, and is equal there to $\phi_k-\id$ and $\phi_k^{-1}-\id$ respectively, using the fact that $f_0^1$ is just a translation there. 

Now remember that on $D_k^\pm$, $\phi_k$ is conjugated to $\Phi_k=\id+\delta_k$ by a homothety of ratio $-v_k^{-1}$. So the $C^k$-norm of $\gamma_k$ is roughly of the order of $\|\delta_k\|_k v_k^{-k+1}$, which is bounded below by $(u_k v_k^{-1}) v_k^{-k+1}=u_kv_k^{-k}$. Now we want this $\gamma_k$ to be $\CC^k$-small in order to get the $C^\infty$-convergence of $(f^1_k)_k$.  This is precisely the purpose of the initial hypothesis on the difference of convergence speed between $(u_k)_k$ and $(v_k)_k$.

Now in the middle of $S_k$, $\f_k-\id=\phi_k-\id$ is not $C^k$- or even $C^2$-small, and neither is $\f_k^{-1} \circ f_0^{1/2} \circ \f_k - f_0^{1/2}$. This is where the size and disymetric shape of $\delta_k$ come into play. Indeed, one can check that on one half of a fundamental interval $f_0^{-p}(J_k^-)$ lying in $S_k$, $\f_k^{-1} \circ f_0^{1/2} \circ \f_k -
f_0^{1/2}$ is precisely $\phi_k - \id$, whose $C^2$-norm there is this time of the order of $\|\delta_k\|_2 u_k^{-1}$ (again by a homothety argument), which is bounded below by $(u_k v_k^{-1}) u_k^{-1}=v_k^{-1}$ which goes to infinity with $k$. 

Thus, superimposing all these perturbations (\emph{i.e} conjugating by $\chi_k = \f_1\circ ... \circ \f_k$, which can be proved to $C^1$-converge, and taking the $\CC^1$-limit) has the desired effect on the time-$1/2$ map of the limit vector field.

\subsubsection{Combination with Anosov--Katok-type methods}
 \label{sss:AK}

We now give the idea of the proof of Theorem \ref{t:liouville} in the case $d=1$. Without loss of generality (replacing $\alpha_1=:\alpha$ by some $\beta\in\Z+\alpha\Z$ if necessary), we can assume that the irrational number $\alpha$ belongs to $(0,1)$. We want to modify the above
construction so that in the end, both $1$ and $\alpha$ are smooth times of the flow of the
limit vector field. The idea is to pick a sequence $(p_k/q_k)_{k}$ of rational approximations of $\alpha$ (not necessarily its convergents), to take an initial vector field $\xi_0$ similar to Sergeraert's (the choice of $(u_k)_{k}$ depending this time on $(q_k)_{k}$, as we will see), and, this time, to ask $\f_k$ to commute almost everywhere not with $f_0^1$ anymore, but with $f_0^{1/q_k}$ (and thus with both $f_0^{p_k/q_k}$ and $f_0^{q_k/q_k} = f_0^1$). More precisely, $\f_k$ is still the identity outside $[x_k^-,x_k^+]$ (defined as before), but this time, on this segment, it is conjugated by the same $\psi$ as before to a diffeomorphism $\Phi_k=\id_\R+\delta_k$ of $\R$ commuting with the translation by $1/q_k$ rather than the unit one. Again, we write $\phi_k=\psi\circ \Phi_k\circ \psi^{-1}$. This time we can check that for every $1\le p\le q_k$, $\gamma^p_k:=f_k^{p/q_k}-f_{k-1}^{p/q_k}$ vanishes outside $[x_k^-,f_0^{-p/q_k}(x_k^+)]$, that it coincides there with $\f_k^{-1} \circ f_0^{p/q_k} \circ \f_k -f_0^{p/q_k}$, and that this vanishes except on $[x_k^\pm,f_0^{-p/q_k}(x_k^\pm)]$, where it is equal to $\phi_k-\id$ and $\phi_k^{-1}-\id$ respectively, again using the fact that $f_0^{p/q_k}$ is just a translation there. 

The restriction to $[0,1/q_k]$ of $\delta_k$ has the same disymetric shape as in the previous paragraph and the same $\CC^0$-norm $u_k/v_k$ (this is to ensure the irregularity of some limit time-$t$ map, just as in \ref{sss:sergeraert}). But this time it is supported in a smaller interval, of length $1/2q_k$.

\begin{figure}[h!]
\centering
\includegraphics[width=6cm]{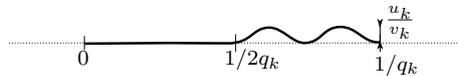}
\put(-90,-5){$\scriptstyle{1/2q_k}$}
\put(-145,-5){$\scriptstyle{0}$}
\put(-30,10){$\scriptstyle{\frac{u_k}{v_k}}$}
\put(-34,-7){$\scriptstyle{1/q_k}$}
\caption{New $\delta_k$}
\label{fig:deltak}
\end{figure}

One can show by repeated applications of the mean value theorem that its $\CC^l$-norm is now bounded below by $u_kv_k^{-1} q_k^l$. Thus $\|\gamma^{p}_k\|_k$, for $p=p_k$ and $q_k$, which is again of the order of $\|\delta_k\|_kv_k^{-k+1}$, is bounded below by $u_k (q_kv_k^{-1})^k$. Now again, we want these $\gamma_k^p$ to be $\CC^k$-small (which requires $u_k$ to be less than $(\frac{v_k}{q_k})^k$ this time) in order to ensure, say, that
$$ \norm{f_k^{p/q_k} - f_{k-1}^{p/q_k}}_k = \norm{\gamma^p_k}_k< 2^{-k-1} \quad \text{for } p=p_k \; \text{and}\; q_k.$$
Then if $\alpha$ is close enough to $p_k/q_k$ (roughly speaking, if $|\alpha - p_k/q_k| =
o(\norm{\xi_l}_k^{-1})$ for $l=k$ and $k-1$, assuming these ``norms'' are well-defined), the above implies $\|f_k^\alpha-f_{k-1}^\alpha\|_k<2^{-k}$, say, which ensures the regularity of the limit time-$\alpha$ map (cf.\ Lemma \ref{l:alpha}). Now much as in \ref{sss:sergeraert}, one can see that $\norm{\f_k}_k$ and $\norm{\xi_k}_k$ are big, and more importantly bigger than $q_k^k$. So, basically, in order for the process to converge, we need $|\alpha - p_k/q_k|$ to be much smaller than $1/q_k^k$ for all $k$, which means $\alpha$ must be a Liouville number. The existence of a non-$C^2$ flow map is guaranteed by the construction much as in \ref{sss:sergeraert} (cf.\ Proposition \ref{p:dv}). 

In \cite{Ey}, we proved the existence of \emph{some} well-chosen $\alpha$, $(q_k)_k$, $(u_k)_k$ and $(v_k)_k$ (obtained by induction) for which the process indeed converges. The main contribution of this part of the present article is to make all the ``rough" estimates
 above precise, \emph{i.e.}\ to control the size of the perturbations in terms of the \emph{initial data} $(q_k)_k$, and to infer that \emph{any} Liouville number $\alpha$ has a suitable approximation by rational numbers for which the process converges and provides the desired vector field $\xi$ (and similarly for \emph{families} of non simultaneously Diophantine numbers).\medskip

Let us now move on to the complete proof of Theorem \ref{t:liouville}.

\subsection{Turning rational approximations into vector fields}
\label{ss:manu}

What we describe in this section is a ``manufacturing process" which, to \emph{any} increasing sequence of positive integers $(q_k)_{k}$, associates a specific $\CC^1$ vector field $\xi$ on $\R_+$, with a smooth time-$1$ map.\ It will be obtained as a $C^1$-limit of a sequence $(\xi_k)_k$ like the one of the previous paragraph, which will be described explicitly this time. Then (in Sections \ref{ss:poly} and \ref{ss:conv}), we show that for any family of non simultaneously Diophantine numbers $\alpha_1,\dots,\alpha_d$, there is a suitable sequence $(q_k)_{k}$ such that the
vector field $\xi$ associated to $(q_k)_{k}$ has all the additional properties listed in Theorem~\ref{t:liouville}.

Let $(q_k)_{k}$ be any increasing sequence of positive integers (fixed until the end of Section \ref{ss:manu}). In order to produce $\xi$, we must first associate to $(q_k)_{k}$ a number of intermediate objects, the main ones being an initial vector field $\xi_0$, smooth on $\R_+$, and a sequence $(\f_k)_{k}$ of smooth commuting diffeomorphisms of $\R_+$. Those are then used, as explained in the outline, to deform $\xi_0$ gradually to new smooth vector fields
$$\xi_k = \chi_k^*\xi_0 \quad \text{where $\chi_k = \f_1\circ ... \circ \f_k$},$$
which converge in $\CC^1$-topology, and we will define $\xi$ as their limit. 

\subsubsection{Common basis}
\label{sss:common}

Some material used to construct $\xi_0$ is independent of $(q_k)_{k}$, namely the coefficients $(v_k)_{k}$ defined by
$$v_k = 2^{-(k+3)^2} \quad\text{for all $k \ge 1$,}$$
and two smooth functions $\beta, \delta \from \R \to [0,1]$ satisfying the following conditions:
\begin{itemize}
\item
$\beta$ vanishes outside $\left[ -\frac 1 8, \frac 5 8 \right]$, equals $1$ on 
$\left[ 0, \frac 1 2 \right]$, and $\norm{\beta}_1 < 16$;
\item $\delta$ vanishes outside $\left[\frac 1 2, 1 \right]$, 
$\delta(x) = \frac12(x-\frac34)^2$ for $x$ close to $\frac34$, and $\norm{\delta}_1 < 1$.
\end{itemize}

\begin{figure}[htbp]
\centering
\includegraphics[width=12cm]{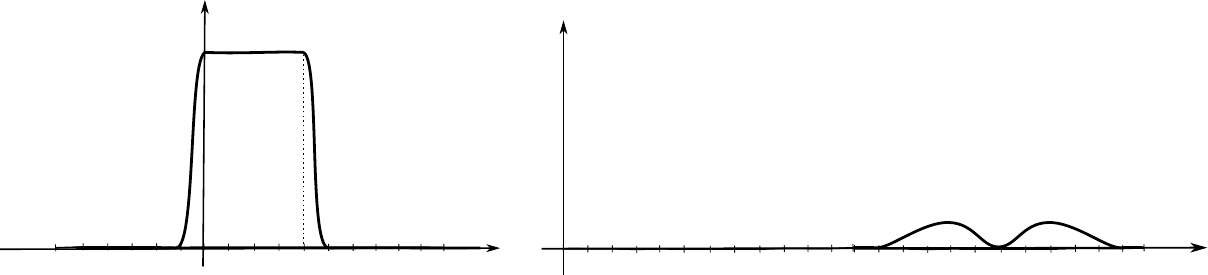}
\put(-72,20){${\scriptstyle \delta}$}
\put(-20,-4){${\scriptstyle 1}$}
\put(-103,-4){${\scriptstyle  \frac 1{2}}$}
\put(-245,20){${\scriptstyle \beta}$}
\put(-248,-4){${\scriptstyle \frac 58}$}
\put(-258,-4){${\scriptstyle \frac 12}$}
\put(-300,-4){${\scriptstyle -\frac 18}$}
\put(-281,67){${\scriptstyle 1}$}
\label{fig:abg}
\end{figure}

\subsubsection{Initial vector field and related objects}
\label{sss:objects}

The coefficients $(u_k)_{k}$ defined now, on the other hand, depend on
$(q_k)_k$:
\begin{equation}\label{e:uk}
u_k = \eta_k \; q_k^{-k} \; v_k^{k} \norm{\delta}_{k}^{-1} \quad
\text{for all $k\ge 1$}
\end{equation}
where $0<\eta_k\le 2^{-k-4}$ is chosen such that, for any $\f\in\Diff^k(\R_+)$,
 \begin{equation}\|\f-\id\|_k\le\eta_k\Rightarrow \|\f^{-1}-\id\|_k\le2^{-k-4}
 \label{e:eta}
 \end{equation}
(such an $\eta_k$ can be obtained using Formula \eqref{e:inverse}). The initial vector field $\xi_0$ is then defined by $\xi_0(0) = 0$, $ \xi_0(x) = -v_1$ for all $x\ge \frac12$ and, for all $k\ge1$,
\begin{equation}\label{e:nu0}
 \xi_0(x) = - v_k + (v_k - v_{k+1})\; \beta (2-2^{k+2} x)
 + (v_k - u_k)\; \beta (2^{k+2} x - 3)\quad \text{for all $x \in [2^{-k-1}, 2^{-k}]$}. 
\end{equation}

\begin{figure}[htbp]
\centering
\includegraphics[width=10cm]{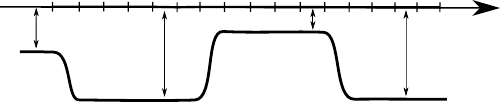}
\put(-262,40){${\scriptstyle v_{k+1}}$}
\put(-267,58){${\scriptstyle 2^{-(k+1)}}$}
\put(-188,25){${\scriptstyle v_{k}}$}
\put(-52,25){${\scriptstyle v_{k}}$}
\put(-35,58){${\scriptstyle 2^{-k}}$}
\put(-101,45){${\scriptstyle u_{k}}$}
\caption{Shape of $\xi_0$.}
\label{fig:nu_0}
\end{figure}
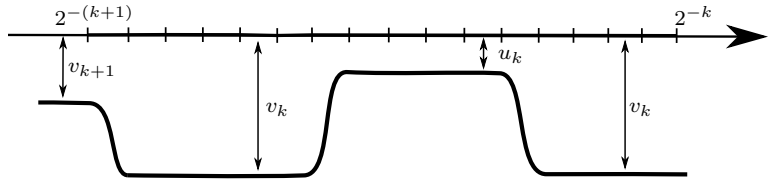

\noindent One easily checks that $\xi_0$ is $C^\infty$, infinitely flat at the origin and $\CC^1$-bounded, with $\Bars{\xi_0}_1 \in (0,1)$ (the choice of $v_k$ is crucial here). Furthermore, $\xi_0$ equals $-v_k$ identically on $D_k^-$ and $D_k^+$ and $-u_k$ on $S_k$ (cf.\ Figure \ref{fig:nu_0}), with $D_k^-=[\frac{17}{16}2^{-k-1}, \frac{22}{16}2^{-k-1}]$, $D_k^+=[\frac78 2^{-k},2^{-k}]$ and $S_k=[\frac{23}{16}2^{-k-1},\frac{27}{16}2^{-k-1}]$, each of length greater than $2^{-k-3}$. 

We denote by $(f_0^t)_{t \in \R}$ the flow of $\xi_0$, and fix a forward orbit $\{ a_l=f_0^l(1), l \ge 0 \}$ of $f_0 = f_0^1$. A simple computation of travel time at constant speed shows that for every $k \ge 1$, there exist integers $i,j$ and $l$ such that
\begin{align} \label{e:ain}
[ a_{j+2} , a_{j-1} ]\subset D_k^-,\quad [ a_{l+2} , a_{l-1} ]\subset D_k^+\quad\text{and}\quad 
[ a_{i+2} , a_{i-1} ]\subset S_k.
\end{align}
We denote by $j(k)$, $l(k)$ and $i(k)$ the smallest integers $j$, $l$, $i$ satisfying the above property. Thus $\xi_0$ equals $-v_k$ on $[a_{ j(k)+2 }, a_{ j(k)-1 }]$ and $[a_{ l(k)+2 }, a_{ l(k)-1 }]$, and hence $f_0^t$ induces on $[a_{ j(k)+1},\linebreak[1] a_{ j(k)-1 }]$ and $[a_{ l(k)+1},\linebreak[1] a_{ l(k)-1 }]$ the translation by $-tv_k$ for $0 \le t \le 1$. Similarly, $f_0^t$ induces the translation by $-t u_k$ on $[a_{ i(k)+1},\linebreak[1] a_{ i(k)-1 }]$. We define
\begin{align}
\label{e:xn}
x_k^+=a_{j(k)}, \quad x_k^-&=a_{l(k)},\quad y_k=a_{i(k)}, \notag\\
J_k^\pm=[x_k^\pm,f_0^{-1}(x_k^\pm)]=[x_k^\pm,x_k^\pm+v_k]\subset D_k^\pm\quad&\text{and}\quad J_k=[y_k,f_0^{-1}(y_k)]=[y_k,y_k+u_k]\subset S_k.
\end{align}

\subsubsection{Conjugating diffeomorphisms and their properties}
\label{sss:conj}

For all $k\ge 1$, let 
\begin{itemize}
\item $\delta_k$ be the $\frac1{q_k}$-periodic map on $\R$ defined on $[0,\frac1{q_k}]$ by $\delta_k (x) = \frac{u_k}{v_k} \delta (q_k x)$;
\item  $\Phi_k=\id_\R+\delta_k$ (which commutes with the translation by $\frac1{q_k}$ and fixes $\frac1{4q_k}\Z$);
\item  $\phi_k=\psi\circ \Phi_k\circ \psi^{-1}$ (with $\psi:t\in\R\mapsto f_0^t(1)$), which commutes with $f_0^{1/q_k}$ on $\R_+^*$, and fixes the orbit of $1$ under $f_0^{1/4q_k}$ (and \emph{a fortiori} under $f_0^1$) since $\psi$ conjugates $f_0^s$ to $T_s$;
\item $\varphi_k=\phi_k$ on $I_k:=\left[ x_k^-,
x_k^+\right]$ and $\id_{\R_+}$ elsewhere.
\end{itemize}

As explained in the overview, the size and shape of $\phi_k-\id$ will be important to prove the regularity of the limit time-$1$ map (cf.\ Proposition \ref{p:cv}) and the non-regularity of other flow-maps (cf.\ Proposition \ref{p:dv}). To that aim, note that 
\begin{equation*}
\forall t\in[j(k)-1,j(k)] \cup [l(k)-1,l(k)],\quad \psi'(t)=\xi_0(f_0^t(1))=-v_k,
\end{equation*}
so $\psi$ induces a homothety of ratio $-v_k$ between $[j(k)-1,j(k)]$ (resp.\ $[l(k)-1,l(k)]$)
and $[a_{j(k)},a_{j(k)-1}]=J_k^-$ (resp.\ $[a_{l(k)},a_{l(k)-1}]=J_k^+$). 

A consequence is that on $J_k^\pm$ (which is stable under $\phi_k$), for all $m\in [\![0,k]\!]$,
\begin{align}
\label{e:estimphim}
\|D^m(\phi_k-\id)\|_{0,J_k^\pm} 
 = v_k^{-m+1} \|D^m(\Phi_k-\id)\|_{0} 
& = v_k^{-m+1} \|D^m\delta_k\|_{0} \notag\\
& =  v_k^{-m+1}\frac{u_k}{v_k}q_k^m\|D^m\delta\|_0\\
&= \eta_k v_k^{k-m}q_k^{m-k}\|\delta\|_k^{-1} \|D^m\delta\|_0\le \eta_k,\notag
\end{align}
so 
\begin{equation}
\label{e:estimphi}
\|\phi_k-\id\|_{k,J_k^\pm} \le \eta_k \quad\text{and}\quad \|\phi_k^{-1}-\id\|_{k,J_k^\pm} \le 2^{-k-4}
\end{equation}
by definition \eqref{e:eta} of $\eta_k$.\medskip

Similarly, $\psi$ induces a homothety of ratio $-u_k$ between $[i(k)-1,i(k)]$ and $[a_{i(k)},a_{i(k)-1}]=J_k$. As a consequence, since $\Phi_k=\id_\R$ on $[0,\frac1{2q_k}]+\frac1{q_k}\Z$ (cf.\ Figure \ref{fig:deltak} for the shape of $\delta_k$ on $[0,\frac1{q_k}]$), and in particular on $N_k:=i(k)-1+\bigcup_{p=0}^{q_k-1}\left[\frac p {q_k},\frac p {q_k}+\frac1{2q_k}\right] \subset [i(k)-1,i(k)]$,
\begin{align}
\label{e:Nk}
\text{$\phi_k=\id_{\R_+}$ on}\quad N'_k:=\psi(N_k)&=\underbrace{\psi(i(k)-1)}_{y_k+u_k}+\bigcup_{p=0}^{q_k-1}\left[-\tfrac p {q_k}u_k,-(\tfrac p {q_k}+\tfrac1{2q_k})u_k\right]\notag \\ 
&=y_k+\bigcup_{r=0}^{q_k-1} \left[(r+\tfrac12)\tfrac{u_k}{q_k},(r+1)\tfrac{u_k}{q_k}\right].
\end{align}

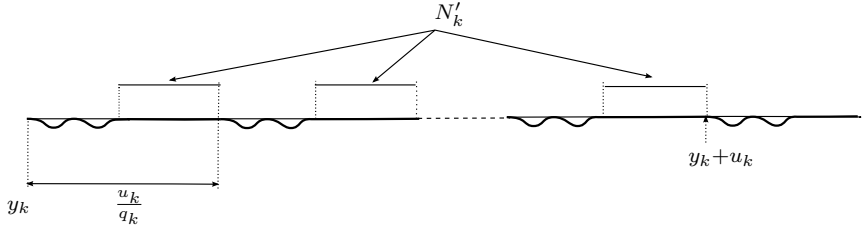
\begin{figure}[htbp]
\centering
\includegraphics[width=11cm]{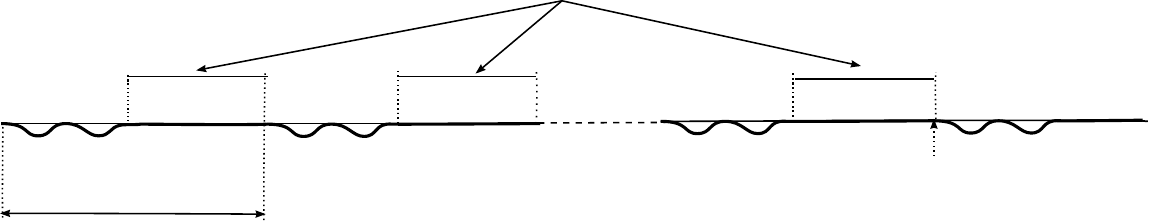}
\put(-160,64){${\scriptstyle N'_k}$}
\put(-280,-8){${\scriptstyle \frac{u_k}{q_k}}$}
\put(-320,-8){${\scriptstyle y_k}$}
\put(-65,10){${\scriptstyle y_k+u_k}$}
\caption{$\phi_k-\id$ near the fundamental interval $[y_k,y_k+u_k]$ of $f_0^1$.}
\label{fig:Nk}
\end{figure}

\subsubsection{Convergence of the time-$1$ maps}
\label{sss:conv}
\def\esti{\mathrm{i}}

We define, for all $k\ge 1$, $\chi_k = \f_1\circ\dots\circ \f_k$, $\xi_k = \chi_k^*\xi_0 = \f_k^*\xi_{k-1}$ and denote by $(f_k^t)_{t\in\R}$ the flow of $\xi_{k}$. Just like in the outline \ref{sss:AK}, for all $p\in[\![0,q_k]\!]$, $\f_k$ being supported in $[x_k^-,x_k^+]$, 
\begin{equation}
\label{e:fkpqk}
f_k^{p/q_k}-f_{k-1}^{p/q_k}=\f_k^{-1} \circ f_{k-1}^{p/q_k} \circ \f_k -f_{k-1}^{p/q_k}\linebreak[1] =0 \text{ outside $[x_k^-,f_{k-1}^{-p/q_k}(x_k^+)]$},
\end{equation}
that is outside $[x_k^-,f_{0}^{-p/q_k}(x_k^+)]\subset B_k$ since $\chi_{k-1}$, which conjugates $f_0^t$ to $f_{k-1}^t$, is supported in $\cup_{j=1}^{k-1}B_j$ which is disjoint from $B_k$ (this fact also implies the first equality below). Now on $[x_k^-,f_{0}^{-p/q_k}(x_k^+)]$,
\begin{equation}
\label{e:f0pqk}
f_k^{p/q_k}-f_{k-1}^{p/q_k}= \f_k^{-1} \circ f_0^{p/q_k} \circ \f_k -f_0^{p/q_k} =\left\{\begin{matrix}
\phi_k-\id&\text{on $[x_k^-,f_0^{-p/q_k}(x_k^-)]$,}\\
\phi_k^{-1}-\id&\text{on $[x_k^+,f_0^{-p/q_k}(x_k^+)]$,}\\
0&\text{on $[f_0^{-p/q_k}(x_k^-),x_+]$.}
\end{matrix}\right.
\end{equation}
Indeed, 
\begin{itemize}
\item for $x$ in $[f_0^{-p/q_k}(x_k^-),x_+]$, $\f_k(x)=\phi_k(x)$, which belongs to the same interval, so $f_0^{p/q_k}(\f_k(x))$ belongs to $[x_k^-, f_0^{p/q_k}(x_k^+)]$ where again $\f_k=\phi_k$, and $\phi_k$ commutes with $f_0^{p/q_k}$,
\item for $x$ in $[x_k^-,f_0^{-p/q_k}(x_k^-)]$, $\f_k(x)=\phi_k(x)$, which belongs to the same interval, $f_0^{p/q_k}(\f_k(x))\le x_k^-$ so it is fixed by $\f_k$ (and its inverse), and  $f_0^{p/q_k}$ is just the translation by $-\frac p{q_k}v_k$ on $[x_k^-,f_0^{-1}(x_k^-)]$, so
\begin{align*}
(\f_k^{-1} \circ f_0^{p/q_k} \circ \f_k -f_0^{p/q_k} )(x)
&= f_0^{p/q_k}(\phi_k(x))-f_0^{p/q_k}(x)\\
&=\phi_k(x)-x\;;
 \end{align*}
 \item similarly, for $x$ in $[x_k^+,f_0^{-p/q_k}(x_k^+)]$, $\f_k(x)=x$, $f_0^{p/q_k}(x)$ belongs to $[f_0^{p/q_k}(x_k^+),x_k^+]$ where $\f_k^{\pm1}=\phi_k^{\pm1}$, which commutes with $f_0^{p/q_k}$, and $f_0^{p/q_k}$ is just the translation by $-\frac p{q_k}v_k$ on $[x_k^+,f_0^{-1}(x_k^+)]$, so
\begin{align*}
(\f_k^{-1} \circ f_0^{p/q_k} \circ \f_k -f_0^{p/q_k} )(x)
&= \phi_k^{-1}(f_0^{p/q_k}(x))-f_0^{p/q_k}(x)\\
&=f_0^{p/q_k}(\phi_k^{-1}(x))-f_0^{p/q_k}(x)\\
&=\phi_k^{-1}(x)-x.
 \end{align*}
\end{itemize}

\begin{proposition}\label{p:cv}
For all $k\ge 1$, 
\begin{equation}
\bigBars{f_k^t - f_{k-1}^t }_k \le 2^{-k-4} \quad \text{for every} \quad t \in
\frac1{q_k} \Z \cap [0,1].\tag{$\esti_k$} \label{ik}
\end{equation}
In particular, the sequence $(f_k^1)_k$ of time-$1$ maps converges in $\Cinf$-topology towards a smooth contraction $f$, whose Szekeres vector field $\xi$ is the $\CC^1$-limit of the sequence of vector fields $(\xi_k)_k$. 
\end{proposition}

\begin{proof}
The estimates follow directly from \eqref{e:fkpqk}, \eqref{e:f0pqk} and \eqref{e:estimphi}, and the convergence of the time-$1$ maps $f_k^1 = f_k$ towards a smooth diffeomorphism $f$ follows. 

Let us check that $f$ is still a contraction. For every $x>0$,
\begin{equation}\label{e:nonnul}
\lrbars{\frac{f_k(x) - f_{k-1}(x)}{f_0(x) - x}}\le 2^{-k-2} \quad\text{for all
$k\ge 1$}.
\end{equation}
Indeed, either $f_k(x)=f_{k-1}(x)$ or, if $x\in J_k^\pm$, $|f_0(x)-x| =  v_k$ and
\begin{align*}
|f_k(x) - f_{k-1}(x)| \le \max(\norm{\phi_k-\id}_{0,J_k^-},\norm{\phi_k^{-1}-\id}_{0,J_k^+}) &=  \norm{\phi_k-\id}_{0,J_k^-} \\
&= u_k\|\delta\|_0 \quad\text{(cf.\ \eqref{e:estimphim})} \\
&\le u_k,
\end{align*}
which implies inequality \eqref{e:nonnul} since $u_k/v_k\le 2^{-k-2}$. 
Thus for all $x \in \R_+^*$,
\begin{align*}
|f(x) - x |  = \left| f_0(x) - x + \sum_{k= 1}^{+\infty} \big(f_k(x) - f_{k-1}(x)\big)\right| 
& \ge |f_0(x) - x| \left(1 - \sum_{k = 1}^{+\infty} 2^{-k-2}\right) \\
& \ge \frac{|f_0(x) - x|} 2 > 0.
\end{align*}
So $f$ has no other fixed point than $0$. 

We could prove the $\CC^1$-convergence of the sequence $(\xi_k)_k$ by hand, as
in \cite{Ey}. But in fact, this convergence can be derived directly from the $\CC^\infty$-convergence of the time-$1$ maps, as an immediate consequence of a theorem by Yoccoz \cite[Chap.\ 4, Theorem 2.5]{Yo} asserting the continuous dependence of the Szekeres vector field with respect to its time-$1$ map (in a more general setting and for suitably defined topologies). 
\end{proof}

\subsubsection{Irregularity of some limit time-$t$ maps}
\label{sss:div}

Let us now define the subsets $H_{k}$, for $k\ge 1$, and $H$ of $[0,1]$ by
\begin{equation*}
H_{k} = \bigcap_{l\ge k} \bigcup_{0 \le p < q_l} \left[ (p+\tfrac14)\tfrac {1}{q_l}, (p+\tfrac34)\tfrac {1}{q_l}\right]
\end{equation*}
and 
\begin{equation}\label{e:Hdef}
H = \bigcup_{k \ge1} H_{k}.
\end{equation}

\begin{proposition}\label{p:dv}
For every $t$ in $H$ (which may be empty!), the time-$t$ map $f^t$ of $\xi$ is not $\CC^2$.
\end{proposition}

\begin{proof}
Let us assume $H$ is nonempty and let $t \in H_{k_0}$ for some $k_0 \ge 1$. We want to prove that $D^2f^{t}$, or equivalently $Nf^t=D\log Df^t=D^2f^t/Df^t$, has no limit at $0^+$. The idea is basically the same as for the irregularity for $t=1/2$ in Sergeraert's construction, even though the consideration of a (possibly) wider set of $t$'s makes the following proof less straightforward. Here again, $D^2f^t$ will take bigger and bigger values in the ``shallow regions'' closer and closer to $0$. 

Let us compute, for every $l \ge k_0 +1$, $Nf^t$ at $c_l = f_0^{-t-\frac1{4q_l}}(y_l)$. First note that 
$$t+\frac1{4q_l}\in \bigcup_{p=0}^{q_l-1} \left[ (p+\tfrac12)\tfrac {1}{q_l}, (p+1)\tfrac {1}{q_l}\right]\subset [0,1]$$
so $c_l$ is simply $y_l + (t+\frac1{4q_l})u_l$, which belongs to $N'_l$ (cf.\ \eqref{e:Nk}) where $\f_l=\phi_l=\id$. Also, let $b_l= y_l +\frac1{4q_l}u_l$, which belongs to the orbit of $1$ under $f_0^{1/4q_l}$ and is thus fixed by $\f_l$. 

Now by invariance of $\xi$ under its flow,
$$Df^t = \frac{\xi \circ f^t}{\xi} \quad \text{on } \R_+^*$$
(here and later, we identify $\xi$ with the map $dx(\xi)$) from which one computes
$$Nf^t = \frac{D\xi \circ f^t - D\xi}{\xi}.$$
In particular,
$$Nf^{t}( c_l ) = - \frac{ D\xi(f^{t}(c_l)) - D\xi(c_l) }{u_{l}}.$$
Now observe that for all $k>l$, $\xi_k$ is the pull-back of $\xi_l$ (resp.\ $f_k^t$ is conjugated to $f_l^t$) by $\f_k\circ\dots\circ \f_{l+1}$ which is the identity on $S_l$ where $c_l$ and $f_l^t(c_l)$ lie, so the above equality becomes
$$Nf^{t}( c_l ) = - \frac{ D\xi_l(f_l^{t}(c_l)) - D\xi_l(c_l) }{u_{l}}.$$
By a similar argument, near the points under scrutiny, $\xi_l=\f_l^*\xi_0= \frac{\xi_0\circ \f_l}{D\f_l}$ (with the same identification \emph{vector field/function} as before). In particular, since $\f_l=\id$ near $c_l$, $\xi_l=\xi_0$ there, and $D\xi_l(c_l)=0$. Furthermore,
$$f_l^t(c_l) = (\f_l^{-1}\circ f_0^t\circ \f_l)(c_l)=\f_l^{-1}(f_0^t(c_l))=\f_l^{-1}(c_l-tu_l)=\f_l^{-1}(b_l)=b_l.$$
Near this point, differentiating $\xi_l= \frac{\xi_0\circ \f_l}{D\f_l}$, we get
$$D\xi_l = D\xi_0 \circ \f_l - (\xi_0 \circ \f_l) \frac{D^2\f_l}{(D\f_l)^2}=u_l\frac{D^2\phi_l}{(D\phi_l)^2}.$$
In particular, using the fact that $\phi_l$ is conjugated to $\Phi_l$ by a homothety of ratio $-u_l$,
\begin{align*}
D\xi_l(b_l) =u_l \frac{(-u_l)^{-1}D^2\Phi_l}{(D\Phi_l)^2}(i(l)-\tfrac1{4q_l})
&=-\frac{D^2\delta_l}{(1+D\delta_l)^2}(-\tfrac1{4q_l})\\
&=-\frac{u_l}{v_l}q_l^2D^2\delta(\tfrac34)=-u_l\frac{q_l^2}{v_l}.
\end{align*}
In the end,
$$Nf^{t}(c_l) = \frac{q_l^2}{v_{l}} \xrightarrow[l\to+\infty]{} + \infty$$
so $f^{t}$ is not $\CC^2$ at $0$.
\end{proof}

\subsection{Polynomial control of the manufactured objects}
\label{ss:poly}

\begin{proposition}\label{l:klemma}
There are maps $n$ and $c \from \N^2 \to \N^*$ such that for any increasing sequence $(q_k)_{k}$ of positive integers, the vector fields $(\xi_k)_{k \ge 0}$ built from $(q_k)_{k}$ and their flows $(f_k^t)_{t \in\R}$ satisfy
\begin{equation*}
\norm{\xi_k \circ f_k^t}_r \le c(k,r) q_k^{n(k,r)} \quad \text{for every $(k,r) \in
\N^2$ $($with $q_0:=1)$ and every $t\in[-1,1]$.}
\end{equation*}
\end{proposition}

This proposition relies on the following assertions.

\begin{lemme}\label{l:nu0}
There are universal bounds on all derivatives of $\xi_0$ and $f_0^t$, $t\in [-1,1]$, \emph{i.e.}\ bounds which depend neither on $(q_k)_k$ nor on $t$. In particular, $\|\xi_0\|_1<1$ and for every $t\in[-1,1]$, $\|Df_0^t\|_0<e$.
\end{lemme}

\begin{lemme}\label{l:gk}
There is a polynomial \up(in $q_k$\up) control on the growth of the derivatives of $\f_k$, \emph{i.e.}\ there exist universal maps $c, n \from \N^*\times \N \to \N^*$ such that for any $(q_k)_{k}$, the associated $(\f_k)_{k}$ satisfies
\begin{equation*}
\max \left( \Bars{\f_k - \id}_r, \Bars{\f_k^{-1} - \id}_r \right)< c(k,r) q_k^{n(k,r)}\quad\text{$\forall(k,r) \in \N^* \times \N$.}
\end{equation*}
\end{lemme}

\begin{proof}[Proof of Proposition \ref{l:klemma} using Lemmas \ref{l:nu0} and
\ref{l:gk}]
We proceed by induction on $k$. Step $k=0$ follows directly from Lemma \ref{l:nu0} and Fa\`a di Bruno's Formula. For $k \ge 1$, step $k$ follows from step $k-1$ and Lemma \ref{l:gk} applying Fa\`a di Bruno's formula and the chain rule to the relations
$$\xi_k = \f_k^*\xi_{k-1} = (\xi_{k-1} \circ \f_k)(D\f_k^{-1} \circ \f_k) \quad 
\text{and} \quad f_k^t = \f_k^{-1} \circ f_{k-1}^t \circ \f_k.$$
\end{proof}

\begin{proof}[Proof of Lemma \ref{l:nu0}]
It is rather clear from the definition \eqref{e:nu0} of $\xi_0$ that $\xi_0$ and its derivatives
are bounded independently of the coefficients $(u_k)_k$, and thus of $(q_k)_k$. We already noted, when we defined $\xi_0$, that $\|\xi_0\|_1<1$. Similar bounds on the derivatives of the flow maps (for a compact set of times) are then obtained from the equalities $\frac d {dt} D^rf_0^t(x) = D^r(\xi_0\circ f_0^t)(x)$, $r\in\N$, applying Fa\`a di Bruno's formula to the second term, bounding the resulting terms by induction except $((D\xi_0)\circ f_0^t)D^rf_0^t$, and concluding with an appropriate version of Gronwall's Lemma. In particular, from $\frac d {dt} Df_0^t(x) = (D\xi_0\circ f_0^t) Df_0^t$ and the bound on $\|\xi_0\|_1$, one gets the desired bound on $\|Df_0^t\|_0$ for $t\in[-1,1]$.
\end{proof}

\begin{proof}[Proof of Lemma \ref{l:gk}]
Let $k \ge 1$. Recall that $\f_k-\id$ vanishes outside $[x_k^-, x_k^+]$ and is equal there to $\psi\circ \Phi_k\circ \psi^{-1}-\id$. In particular, $\|\f_k-\id\|_0$ is at most $x_k^+\le 2^{-k}$ (recall $x_k^+\in D_k^+\subset [0,2^{-k}]$). 

Now let $x\in [x_k^-, x_k^+]$. Observe that $\f_k$ has a fixed point between $f_0^1(x)$ and $x$ (since it fixes the orbit $(f_0^p(1))_{p\in\Z}$). Thus $\f_k(x)$, which is less than or equal to $x$, must lie in $[f_0^1(x),x]$, \emph{i.e.}\ be of the form $f_0^s(x)$ for some $s\in[0,1]$. 

Now recall $\psi(t)=f_0^t(1)$, so $D\psi = \xi_0\circ \psi$ and $D\psi^{-1} = \frac1{\xi_0}$. Hence on $[x_k^-, x_k^+]$, the chain rule gives:
\begin{equation}
D\f_k = \D\psi(\Phi_k\circ \psi^{-1})\times D\Phi_k(\psi^{-1})\times D\psi^{-1} = \frac{\xi_0\circ \f_k}{\xi_0} \times D\Phi_k\circ \psi^{-1}.
\label{e:Dfk}
\end{equation}
In particular,
$$\frac{\xi_0\circ \f_k}{\xi_0}(x)=\frac{\xi_0\circ f_0^s}{\xi_0}(x)=Df_0^s(x),$$
so
\begin{align*}
|D\f_k(x)-1|&\le  |\tfrac{\xi_0\circ \f_k}{\xi_0}(x)-1|\times\|D\Phi_k\|_0+ \|D(\Phi_k-\id)\|_0\\
&\le \|Df_0^s-1\|_0(1+\|D\delta_k\|_0)+\|D\delta_k\|_0\\
&\le e\left(1+ q_k\frac{u_k}{v_k}\right)+ q_k\frac{u_k}{v_k}\\
\end{align*}
which can be checked to be less than $2^{-k}$. Note that if we get a polynomial (in $q_k$) control on the growth of the derivatives of $\f_k-\id$, this last estimate automatically gives one on $\f_k^{-1}-\id$ thanks to Formula \eqref{e:inverse}.

Now the polynomial control we wish on higher derivatives of $\f_k$ is obtained by induction on the degree of derivation using \eqref{e:Dfk}. Thanks to the chain rule and Fa\`a di Bruno's formula, it is enough to prove such a control on the $r$-norm of each piece (other than $\f_k$), namely: $\xi_0$, $\frac1{\xi_0}$ and $\psi^{-1}$ on $[x_k^-,x_k^+]$, and $D\Phi_k$ on $\R$. We already dealt with $\xi_0$. As for $\frac1{\xi_0}$, for all $r\ge 1$, a simple induction gives :
\begin{equation}\label{e:Df0p}
D^{r+1}(\tfrac1{\xi_0}) = \frac{Q_r(\xi_0,...,D^r\xi_0)}{\xi_0^{2^r}},
\end{equation}
where $Q_r$ is a universal polynomial (independent of $\xi_0$) in $r+1$ variables. According to Lemma \ref{l:nu0}, for each $r$, the numerator of \eqref{e:Df0p} is bounded independently of $(q_k)_k$. As for the denominator, $|\xi_0(x)| \ge u_k$ for all $x \in [x_k^-,x_k^+]$, so by definition \eqref{e:uk} of $u_k$,
$$\frac{1}{\xi_0^{2^r} } \le \left(\eta_k^{-1}v_k^{-k} 
\norm{\gamma}_{k} \right)^{2^r} q_k^{2^r(k+1)},$$
which is the kind of control we were looking for to prove Lemma \ref{l:gk}. This also settles the case of $\psi^{-1}$ since $D\psi^{-1}=\frac1{\xi_0}$. Finally, we already saw that 
$$\|D^r(\Phi_k-\id)\|_0 \le \eta_k v_k^{k-1}q_k^{r-k}$$
which concludes the proof.
\end{proof}

\subsection{Convergence of the time-${\alpha_i}$ maps and existence of non-$\CC^2$ time-$t$ maps}
\label{ss:conv}

\begin{proposition}\label{p:principal}
Let $\alpha_1,\dots,\alpha_d$, $d\in\N^*$, be non simultaneously Diophantine irrational numbers. Then there is a sequence $(q_k)_{k}$
of positive integers such that the vector field $\xi$ built
from $(q_k)_{k}$ has all the properties described in Theorem
\ref{t:liouville}.
\end{proposition}

Let $\alpha_1,\dots,\alpha_d$ be as in the statement, with the additional harmless assumption that $\alpha_i\in(0,1)$ for all $i$. By definition of (non) simultaneously Diophantine, there exists a sequence
$(q_k)_{k}$ of positive integers satisfying
\begin{equation}\label{e:C}
\max(\|q_k\alpha_1\|,\dots,\|q_k\alpha_d\|)< \frac{2^{-k-2} c(k,k)^{-1} }{q_k^{n(k,k)-1}}=:q_k\eps_k
\quad \text{for all $k\ge 1$} \tag{$C_k$}
\end{equation}
(where $c$ and $n$ are the maps given by Proposition \ref{l:klemma}), with the additional requirement that
\begin{equation} \label{e:C'}
\frac1{q_{k+1}}<\eps_k\quad\text{for all $k\ge 1$}.\tag{$C_k'$}
\end{equation}
This last condition ensures that every segment $\left[\frac{p}{q_k}- \eps_k, \frac{p}{q_k}+\eps_k\right]$, $p\in\Z$, contains at least two elements of $\frac{1}{q_{k+1}}\Z$, making 
\begin{equation}\label{e:K}
K= \bigcap_{k\ge 1}\bigcup_{0\le p\le q_k} \left[\frac{p}{q_k}- \eps_k, \frac{p}{q_k}+\eps_k\right]
\end{equation}
a Cantor set, with $\alpha_1,\dots,\alpha_d\in K$ thanks to \eqref{e:C}. Similarly, for such a sequence $(q_k)_k$, the set $H$ defined by \eqref{e:Hdef} is a Cantor set (in particular nonempty). Hence, Proposition \ref{p:principal}, and thus Theorem \ref{t:liouville}, follow from Lemma \ref{l:alpha} below and Proposition \ref{p:cv}.
\begin{lemme}\label{l:alpha}
Let $\alpha_1,\dots,\alpha_d$, $d\in\N^*$, be non simultaneously Diophantine irrational numbers, $(q_k)_{k}$ a sequence
of positive integers satisfying \eqref{e:C} and \eqref{e:C'} for all $k \ge 1$, and $K$ the Cantor set defined by \eqref{e:K}. Then the sequence $(\xi_k)_k$ of vector fields associated to $(q_k)_k$ and their flows satisfy
\begin{equation*}
\bigBars{f_k^\tau - f_{k-1}^\tau }_k \le 2^{-k} \quad \text{for every $k\ge 1$ and $\tau\in K$}.
\end{equation*}
As a consequence, the time-$\tau$ map of the limit $\xi$ of $(\xi_k)_k$ is smooth for every $\tau\in K$.
\end{lemme}

\begin{proof}
Let $\tau\in K$ and $(r_k)_{k}$ be the unique sequence of integers such that 
\begin{equation*}
\tau\in  \left[\frac{r_k}{q_k}- \eps_k, \frac{r_k}{q_k}+\eps_k\right]\quad\text{for every $k\ge 1$}.
\end{equation*}
For all $k\ge 1$,
$$ \bigBars{ f_k^\tau - f_{k-1}^\tau }_k \le \bigBars{ f_k^\tau -
f_{k}^{r_k/q_k} }_k + \bigBars{ f_k^{r_k/q_k} - f_{k-1}^{r_k/q_k} }_k +
\bigBars{f_{k-1}^{r_k/q_k} - f_{k-1}^\tau }_k.$$
According to \eqref{ik} in Proposition \ref{p:cv}, the central term is less than
$2^{-k-4}$. Now for all $n\in[\![0,k]\!]$,
$$D^n\left( f_k^\tau - f_{k}^{r_k/q_k} \right) = D^n\left(\int_{r_k/q_k}^\tau 
\frac{df_k^t}{dt} dt \right)  
= \int_{r_k/q_k}^\tau D^n(\xi_k \circ f_k^t) dt,  $$
so 
$$ \bigBars{f_k^\tau - f_{k}^{r_k/q_k} }_k \le \left| \tau -
\frac{r_k}{q_k} \right| \bigBars{\xi_k \circ f_k^t }_k \le 2^{-k-2}$$
according to \eqref{e:C} and Proposition \ref{l:klemma}. A similar argument gives 
$$ \bigBars{f_{k-1}^{r_k/q_k}  - f_{k-1}^\tau}_k \le 2^{-k-2}$$
and in the end,
\begin{equation*}
\bigBars{f_k^\tau - f_{k-1}^\tau }_k \le 2^{-k}.
\end{equation*}
\end{proof}

\bigskip

\begin{footnotesize}
\noindent {\bf H\'el\`ene Eynard-Bontemps}

\noindent \textbf{Institut Fourier}, UMR 5582, Laboratoire de Math\'ematiques

\noindent  \textbf{Universit\'e Grenoble Alpes}, CS 40700, 38058 Grenoble cedex 9, France

\noindent \texttt{helene.eynard-bontemps@univ-grenoble-alpes.fr }

\end{footnotesize}

\end{document}